\documentclass{amsart}
\usepackage{amssymb}
\usepackage{graphics} 
\usepackage{epsfig} 
\usepackage{graphicx}
\usepackage{epstopdf}
\usepackage[colorlinks=true]{hyperref}
\hypersetup{urlcolor=blue, citecolor=red}
\usepackage{geometry}
\numberwithin{equation}{section}    
\newtheorem{assumption}{Assumption}[section] 
\newtheorem{theorem}{Theorem}[section]
\newtheorem{corollary}[theorem]{Corollary}

\newtheorem{lemma}[theorem]{Lemma}
\newtheorem{proposition}[theorem]{Proposition}

\theoremstyle{definition}

\newtheorem{remark}{Remark}

\title[Vanishing viscosity limit to vortex sheet ]
      {Vanishing viscosity limit to vortex sheet\\
      for the isentropic compressible \\
      circularly symmetric 2D flow }

\author[Helong Lu,]{}

\email{\texttt{luhelong1988@126.com, }}

 \email{}
 \email{}

\begin{document}
\maketitle

\centerline{\scshape Helong Lu }

\medskip
{\footnotesize
}

\medskip

\bigskip
\begin{abstract}
  In this paper, we consider the small viscosity limit problem for the isentropic compressible Navier-Stokes equations in a 2D exterior domain with impermeable boundary conditions , and the corresponding Euler equations  have  vortex sheet solutions.
  We obtain that away from the boundary and the contact discontinuous the isentropic compressible viscous flow can be approximated by the corresponding inviscid flow, near the boundary (the contact discontinuous) there is a boundary layer (vortex layer)for the angular velocity in the leading order expansion of solution, while the radial velocity and the pressure do not have boundary layers (vortex layers) in the leading order. We rigorously justify the asymptotic behavior of solutions in the $L^{\infty}$ space for the small viscosities limit in the Lagrangian coordinates.

\textbf{key words.} {boundary layer, vortex sheet, vortex layer, compressible viscous flows, circularly symmetric}
\end{abstract}

\section{ Introduction}
We consider the compressible viscous flow in 2-D domain.
\begin{equation}\label{eq1}
\begin{cases}
\partial_t\rho^{\epsilon} + \rm div(\rho^{\epsilon} \mathbf{u}^{\epsilon})&=0,\\
\\
\rho^{\epsilon}\{\partial_t \mathbf{u}^{\epsilon}+ (\mathbf{u}^{\epsilon}\cdot\nabla)\mathbf{u}^{\epsilon}\}+ \nabla p^{\epsilon}
&=\epsilon\{\mu\Delta \mathbf{u}^{\epsilon} +(\lambda +\mu )\nabla(\rm div  \mathbf{u}^{\epsilon})\} .
\end{cases}.
\end{equation}
where $\rho^{\epsilon}, \ p^{\epsilon}= \frac{1}{\gamma}(\rho^{\epsilon})^{\gamma}$, and
$ \mathbf{u}^{\epsilon}=(u^{\epsilon},v^{\epsilon})^T$  denote the fluid density, the pressure and the velocity, respectively;
$\lambda$ and $ \mu$ are constants viscosity coefficients, $\lambda + \mu>0$ and $\mu> 0$.

Let $\Omega \subseteq \mathbf{R}^2$, and suppose that the flow is occupied in $\Omega$,
$(\rho^{\epsilon}, \mathbf{u}^{\epsilon}) \mid_{t=0} =(\rho_0,\mathbf{u}^{\epsilon}_0)(x)$.\\
and the boundary condition
\begin{equation}\label{b1}
 \mathbf{u}^{\epsilon}\cdot n=0,\ \ \ \mathbf{u}^{\epsilon}\cdot \tau =u^{\epsilon}_{\tau},  \ \ \ t>0, x\in \partial\Omega.
\end{equation}

We are interested in the asymptotic behavior of the flow described by the problem \eqref{eq1}-\eqref{b1},when the viscosity coefficients tends to zero.
Formally let $\epsilon\rightarrow 0$,  we have Euler system:
\begin{equation}\label{eq2}
\begin{cases}
\partial_t\rho+ \rm div(\rho \mathbf{u})&=0,\\
\rho(\partial_t \mathbf{u}+ \mathbf{u}\cdot\nabla\mathbf{ u})+ \nabla p &=0,\\
(\rho,\mathbf{u})|_{t=0}&=(\rho_0, \mathbf{u}_0)(x),\\
\mathbf{u}\cdot n&=0.
\end{cases}
\end{equation}
\\
and there isn't any constraint on the tangential velocity field on the boundary $\partial\Omega$ for the compressible inviscid flow. Thus, it is quiet different from the boundary conditions \eqref{b1} for viscous flow. This kind of inconsistent boundary conditions  between the viscous flow and the inviscid flow gives rise to a very thin layer near the boundary for the small viscosities limit,and this layer is known as the boundary layer,in which the behavior of flow  changes dramatically. To study the behavior of boundary layers is a very interesting and classical problem in the fluid mechanics.

the boundary layer problem for incompressible flow with nonslip boundary condition was studied by Prandtl first in 1904. Prandtl \cite{Prandtl} studied the small viscosity limit for the incompressible Navier-Stokes equations with the nonslip boundary condition and formally derived that the boundary layer is described by a degenerate parabolic-elliptic coupled system which is so-called Prandtl equations. There have been many interesting results on the well-posedness or ill-posedness of the Prandtl equations, one can see \cite{Alex,WE2,G3,G5,Guo, Hong, Mas, Ole,zp} ,for instance, and there also some works devoted to the validity of boundary layer theory, see \cite{Mas1,Sam1,Sam2,TW}, for instance.

There are a few results about the small viscosity limit for the compressible flow, as it is even more complicated than the incompressible case, see \cite{LW,XY,WX,FR,WW}.

We are interest in circularly symmetric solution of the form:
$$ \rho^{\epsilon}(t,x)= \tilde{\rho}^{\epsilon}(t,r),\ \ \
\mathbf{u}^{\epsilon}(t,x)=(\tilde{u}^{\epsilon},\tilde{v}^{\epsilon})^T (t,r),\  \
p^{\epsilon}(t,x)=\tilde{p}^{\epsilon}(t,r)\ \ t>0, r> a.$$
where $u^{\epsilon}, v^{\epsilon} $ denote the radial and angular component of the velocity,respectively; and the compressible inviscid flow admits a vortex sheet solution.

For the inviscid flow containing shock or rarefaction waves, see \cite{GX,HWY,X1,Y,ZHH,ZWT},for instance.
In the case that the solutions to the Euler system containing contact discontinuity is much more subtle,  see \cite{HLM,HMX,M}, for instance.
Here  we consider the compressible inviscid flow admits a vortex sheet solution, but ignore the initial layer by construct a special initial smooth data for the viscous flow.

This paper is organized as follows. First, in section 2, we prove the local existence of vortex sheet, and then in section 3, we rewrite the problem in lagrangian coordinates and then formally derive asymptotic expansions of solutions w.r.t.$\epsilon$ and deduce problems of outer expansion profiles, boundary layer profiles and vortex layer profiles by multi-scale analysis. in section 4, we give the well-posedness of problems for boundary layer profiles, vortex layer profiles and outer expansion profiles of solutions. In section 5, we study the stability of vortex layers and boundary layers, and then rigorously justify the asymptotic expansions for the small viscosity limit.
\\
\section{local existence of vortex sheet}
We rewrite the equation \eqref{eq1} in circularly symmetric form:
\begin{equation}\label{eq3}
\begin{cases}
\partial_t\rho^{\epsilon} + \partial_r(\rho^{\epsilon} u^{\epsilon})+ \frac{\rho^{\epsilon}}{r}u^{\epsilon}=0,\\
\rho^{\epsilon}(\partial_t u^{\epsilon} + u^{\epsilon}\partial_r u^{\epsilon})+\partial_r p^{\epsilon} - \frac{\rho^{\epsilon}}{r}(v^{\epsilon})^2=
\epsilon (\lambda+2\mu)\{\partial^2_r u^{\epsilon}+ \frac{1}{r}(\partial_r u^{\epsilon} - \frac{u^{\epsilon}}{r})\},\\
\rho^{\epsilon}(\partial_t v^{\epsilon} + u\partial_r v^{\epsilon}) + \frac{\rho^{\epsilon}}{r}u^{\epsilon}v^{\epsilon}=
\epsilon\mu\{\partial^2_r v^{\epsilon}+ \frac{1}{r}(\partial_r v^{\epsilon} - \frac{v^{\epsilon}}{r})\},\\
\end{cases}
\end{equation}
 Formally we have the compressible inviscid flow, described by the Euler system:
 \begin{equation}\label{eq4}
\begin{cases}
\partial_t\rho + \partial_r(\rho u)+ \frac{\rho}{r}u&=0,\\
\rho(\partial_t u + u\partial_r u)+\partial_r p - \frac{\rho}{r}v^2&=0,\\
\rho(\partial_t v + u\partial_r v) + \frac{\rho}{r}uv&=0,\\
\end{cases}
\end{equation}
At first we consider the system \eqref{eq4} admits a piecewise smooth solution.Then \eqref{eq4} can written as the balance laws£º
\begin{equation}\label{eq5}
\begin{cases}
\partial_t(r\rho)+\partial_r(r\rho u)&=0,\\
\partial_t(\rho u)+\partial_r(\rho u^2+ p)- \frac{\rho}{r}(v^2-u^2)&=0,\\
\partial_t(\rho v)+\partial_r(\rho uv) + \frac{2\rho}{r}uv&=0,\\
\end{cases}
\end{equation}
For a piecewise smooth solution of  \eqref{eq5} :
\[U=\begin{cases}
U_{-}=(\rho_{-},u_{-},v_{-})(t,r)&\ a<r<\varphi(t),\\
\\
U_{+}=(\rho_{+},u_{+},v_{+})(t,r)&\ r>\varphi(t),
\end{cases}\]
on the front $ r=\varphi(t)$, the Rankine-Hugoniot conditions holds:
\[\begin{cases}
\varphi'[\rho] &=[\rho u],\\
\varphi'[\rho u]&=[(\rho u^2+p)],\\
\varphi'[\rho v]&=[\rho uv],
\end{cases}
\]
where the bracket $[\cdot]$ stands for the jump of the associated function across the front.

Suppose that $m=\rho(u-\varphi')=0$,\ on $r= \varphi(t)$, i.e., no mass transfer flux across the front, this corresponds to vortex sheet, on the front $r=\varphi(t):$
$$ [v]\neq 0, \quad \quad [p]=0=[u].$$

Therefore, we have the following free boundary problem:
\begin{equation}\label{eq6}
\begin{cases}
U_{-},U_{+}\  \text { satisfy \eqref{eq5} in classical sense},\\
u_{-}=0,\quad r=a,\\
\rho_{+}=\rho_{-},\ \ u_{+}=u_{-}=\varphi', \ \quad  r= \varphi(t),\\
U_{\pm}|_{t=0}=(\rho_{\pm,0},u_{\pm,0},v_{\pm,0})(r).
\end{cases}
\end{equation}

In order to study the local existence of vortex sheet, first using the following coordinate transform to straighten the free boundary.
\[ \begin{cases}
t=\widetilde{t}\\
\\
r=\Phi(\widetilde{t},\widetilde{r})=a+\frac{\varphi(\widetilde{t})-a}{\varphi(0)-a}(\widetilde{r}-a)
\end{cases}
\]\\

Let $U=(\rho, u ,v)^{T}$, $c^2=p'(\rho)$,
drop the tildes, then
\begin{equation}\label{eq7}
\begin{cases}
\partial_t U +B(U,\varphi)\partial_r U + C(U,\varphi)= 0,\\
u_{-}= 0,\quad \text{on}\ \  r= a \\
\rho_{+}=\rho_{-},u_{+}=u_{-}=\varphi' \  \ \text{on}\ \  r= b\\
U|_{t=0}=U^{0}(r)=(\rho_0,u_0,v_0)(r),\\
\end{cases}
\end{equation}
where $b \triangleq \varphi(0)$,
\[
B=
\frac{1}{\Phi_r}
\left(
\begin{array}{ccc}
(u-\Phi_t) & \rho  & 0\\
\\
\frac{c^2}{\rho} &(u-\Phi_t) & 0\\
\\
0&0& (u-\Phi_t)\\
\end{array}
\right),
\ \ \ \ \ \ \ \
C=\frac{1}{\Phi}
\left(
\begin{array}{ccc}
\rho u\\
\\
- v^2\\
\\
uv\\
\end{array}
\right)
\]
\\
and the corresponding compatibility conditions holds.

For the  problem \eqref{eq7}, we take the following iteration scheme:
\begin{equation}\label{eq17}
\begin{cases}
\partial_tU_{\pm}^{(n+1)} + B(U_{\pm}^{(n)},\varphi^{(n)})\partial_r U_{\pm}^{(n+1)} + C(U_{\pm}^{(n)},\varphi^{(n)})=0,\\
U_{\pm}^{(n+1)}|_{t=0}= U_{\pm}^0(r),\\
U_{+,1}^{(n+1)}=U_{-,1}^{(n+1)},U_{+,2}^{(n+1)}=U_{-,2}^{(n+1)}, \  \ \ r=b,\\
U_{-,2}^{(n+1)}=0,\ \ \ r=a,
\end{cases}
\end{equation}
 starting with $U^{(0)}_{\pm}= U^0_{\pm}$, and where
$\varphi^{(n)}(t)=b+ \int_0^t U_{-,2}^{(n)}(s,b)ds$.

If we prove  $\{U^{(n)}\}$ is a Cauchy sequence and uniformly bounded in $C^1$ norm, then the system \eqref{eq7} has a unique  solution.

We shall adapt Corli's idea \cite{C} to prove this result.

So, we consider the linearized problem:
\begin{equation}\label{eq8}
\begin{cases}
\partial_t w + B(U,\varphi)\partial_r w = f ,\\
w_{+,1}=w_{-,1},\ w_{+,2}= w_{-,2},\  \  \ \  r=b,\\
w_{-,2}=0, \ \ \ \  r=a,\\
w|_{t=0}= w^{0}(r).
\end{cases}
\end{equation}
where $\varphi(t)= b + \int_0^t U_{-,2}(s,b)ds.$

We know that the eigenvalues of $B(U,\varphi)$ are
$$ 
  \lambda_1 =\frac{1}{\Phi_r}( u-\Phi_t-c), \  \lambda_2 = \frac{1}{\Phi_r}(u-\Phi_t), \   \lambda_3 = \frac{1}{\Phi_r}(u-\Phi_t + c). $$
then the corresponding right eigenfunctions are:
$$ r_1=(\rho ,-c,0)^{T},\ \ r_2=(0,0,1)^{T}, \ \ r_3=(\rho ,c,0)^{T}.$$
Let$R=(r_1,r_2,r_3)$, we got
\[R^{-1}=(l_1,l_2,l_3)^{T}=\frac{1}{2\rho c}\cdot\left(
\begin{array}{ccc}
c & -\rho  & 0\\
0 & 0 & 2\rho c\\
c & \rho  &0\\
\end{array}
\right)
\]
Let $\widetilde{w}=R^{-1} w$, since $ (\partial R^{-1})R =-R^{-1}\partial R $, then we have:
\begin{equation}\label{eq9}
\begin{cases}
\partial_t\widetilde{w} +\Lambda(U,\varphi)\partial_r \widetilde{w}=
m \widetilde{w} +\widetilde{f},\\
[\widetilde{w_1}]=0, [\widetilde{w_3}]=0,\ \ \ \ r=b,\\
\widetilde{w_1}= \widetilde{w_3},\ \ \  \ r=a,\\
\widetilde{w}|_{t=0}=R^{-1}w^{0}(r):=\widetilde{w}^0(r).
\end{cases}
\end{equation}
where
$ m= (\partial_t R^{-1} + \Lambda(U,\varphi)\partial_r R^{-1})R$, \ \
$\widetilde{f}= R^{-1}f$,\  $\Lambda=R^{-1}BR$.\\
In order to  consider the problem \eqref{eq9},we first  consider the diagonal linear problem
\begin{equation}\label{eq10}
\begin{cases}
\partial_t w+\Lambda(U,\varphi)\partial_r w=f,\\
[w_1]=0,[w_3]=0,\ \ \ \ \ r=b,\\
w_1=w_3, \ \ \ \ r=a,\\
w|_{t=0}=w^0(r).
\end{cases}
\end{equation}
We assume that two compatibility conditions hold for \eqref{eq10}.

Denote by $\Gamma_k (s;t,r)= (s,\gamma_k (s;t,r)) $ the characteristic curve of the operator $ \partial_t +\lambda_k\partial_r$ passing through $(t,r)$ at time $s=t,$ i.e.,
\[\begin{cases}
\frac{d\gamma_k(s;t,r)}{ds}=\lambda_k(s,\gamma_k (s;t,r)),\\
\\
\gamma_k(t;t,r)=r.
\end{cases}
\]
\\
We denote
\begin{align*}
&\Omega^{\pm}_{T}=\{ (t,r): \pm(r-b)>0 \cap\{r\geq a\},0< t <T\},\\
&\Omega^{+}_1 = \{(t,r)\in\Omega^{+}_{T};r\geq \gamma^{+}_3(t;0,b)\},\\
&\Omega^{-}_1 =\{(t,r)\in\Omega^{-}_{T};\gamma^{-}_3(t;0,a)\leq r \leq  \gamma^{-}_1(t;0,b)\},\\
&\Omega_2 = \{(t,r)\in\Omega^{-}_{T};a\leq r \leq  \gamma^{-}_3(t;0,a)\},\\
&\Omega_3 =  \{(t,r)\in\Omega^{-}_{T};\gamma^{-}_1(t;0,b)\leq r \leq  b\},\\
&\Omega_4 =\{(t,r)\in\Omega^{+}_{T};b\leq r\leq\gamma^{+}_3(t;0,b)\},
\end{align*}
where $ \gamma^{\pm}_k = \gamma_k |_{\Omega^{\pm}_T}$.

We can now write  the explicitly  solution of \eqref{eq10}, so in $\Omega^{\pm}_1$
\begin{equation}\label{s1}
w^{\pm}_j=w_j^0(\gamma^{\pm}_j(0;t,r)) +\int_0^t f_j(s;\gamma^{\pm}_j(s;t,r))ds, \ \ \ j=1,2,3.
\end{equation}
In $\Omega_2,$
\begin{align}\label{s2}
\begin{cases}
&w^{-}_1(t,r)=w_1^0(\gamma^{-}_1(0;t,r)) + \int_0^t f_1(s,\gamma^{-}_1(s;t,r))ds, \\
&w^{-}_2(t,r)=w_2^0(\gamma^{-}_2(0;t,r)) + \int_0^t f_2(s,\gamma^{-}_2(s;t,r))ds,\\
&w^{-}_3(t,r)
=w^{-}_1(\tau^{-}_3,a) + \int_{\tau^{-}_3}^t f_3(s,\gamma^{-}_3(s;t,r))ds.
\end{cases}
\end{align}
we have used the boundary condition $w_{1}= w_{3}$ on $ r=a$, and $ 0< \tau^{-}_3(t,r) <t$ is the unique root of $\gamma^{-}_3(\tau^{-}_3;t,r)=a $ for any fixed $(t,r)\in \Omega_2$.

In $\Omega_3\cup\Omega_4,$\\
\begin{align}\label{s3}
\begin{cases}
&w^{\pm}_2(t,r)=w_2^0(\gamma^{\pm}_2(0;t,r)) + \int_0^t f_2(s,\gamma^{\pm}_2(s;t,r))ds,\\
&w^{-}_3(t,r)=w_3^0(\gamma^{-}_3(0;t,r)) + \int_0^t f_3(s,\gamma^{-}_3(s;t,r))ds,\\
&w^{+}_1(t,r)=w_1^0(\gamma^{+}_1(0;t,r)) + \int_0^t f_1(s,\gamma^{+}_1(s;t,r))ds, \\
&w^{-}_1(t,r)
=w_1^{+}(\tau^{-}_1,b)+\int_{\tau^{-}_1}^t f_1(s,\gamma^{-}_1(s;t,r))ds,\\
&w^{+}_3(t,r)
=w^{-}_3(\tau^{+}_3,b)+\int_{\tau^{+}_3}^t f_3(s,\gamma^{+}_3(s;t,r))ds.
\end{cases}
\end{align}
where $\gamma^{-}_1(\tau^{-}_1;t,r)=b,\ \ \  \gamma^{+}_3(\tau^{+}_3;t,r)=b.$

Then we have the following results.

\begin{proposition}\label{p1}
Let $ U,f,w^0$ be some family of functions bounded in $C^1(\Omega^{\pm}_{T}),C^0(\Omega^{\pm}_{T}), C^0(I^{\pm})$, respectively, where $I^{\pm}=\Omega^{\pm}_{T}\cap\{t=0\}$, for some $T\in (0,T_0]$;assume that the corresponding compatibility condition holds. Then problem \eqref{eq10} has a unique solution $w$ bounded in $C^0(\Omega^{\pm}_{T})$, and such that
$$\|w(t)\|\leq C\left(\|w^0\|+\int_0^t\|f(s)\|ds \right)$$
for some constant C. Where $\parallel \cdot \parallel $ denote the $C^{0}$ norm.

\end{proposition}
\begin{corollary}\label{cor1}
Under the assumption of Proposition \ref{p1},  assume that $ w^0 \in C^1(I^{\pm})$ and  the corresponding compatibility conditions hold. Let $ \alpha_i,\beta_i$ be some functions in $C^1(\Omega^{\pm}_{T}), 1\leq i\leq3$, and take
$$f_i = \alpha_i(\partial_t\beta_i + \lambda_i\partial_r\beta_i)$$
Then the solution $w$ to \eqref{eq10} is in $C^1$.
\end{corollary}
\begin{proof}
The idea can be given in \cite{HW}, for completeness we sketch the idea in the following way.
Since the second factor in $f_i$ is the derivative of $\beta_i$ along the i-characteristic direction.
From \eqref{s1}-\eqref{s3}, the representation of the components of solution can be classified into 2 forms:
\begin{enumerate}
\item[\rm (i)]
 $g (t,r)=g^0(\gamma(0;t,r)) +\int_0^t \alpha(s,\gamma(s;t,r))\frac{d}{ds}\beta(s,\gamma(s;t,r)) ds,$
\item[\rm(ii)]
$g (t,r)=h (\tau(t,r)) +\int_{\tau}^t \alpha(s,\gamma(s;t,r))\frac{d}{ds}\beta(s,\gamma(s;t,r))ds.$
\end{enumerate}
where $0< \tau(t,r) < t$  is $C^1 $ function.

 one suitably integrates by parts and find the derivatives of $g$ through the form $(i)$.i.e.,
\begin{align*}
\partial_r g  =&
\int_0^t \left(\partial_r\alpha(s,\gamma(s;t,r))\frac{d}{ds}\beta(s,\gamma(s;t,r))
- \partial_r\beta(s,\gamma(s;t,r))\frac{d}{ds}\alpha(s,\gamma(s;t,r))  \right)ds\\
&+ (\alpha \partial_r \beta)(t,r)-(\alpha \partial_r \beta)(0,\gamma(0;t,r))+  \partial_r g^0(\gamma(0;t,r)) .
\end{align*}
In the same way one finds $\partial_t g$.
Similarly, one can check the smoothness for the form $(ii)$.
\end{proof}
\begin{proposition}\label{p2}
Let $ U,f,w^0 $ as in Proposition \ref{p1}, and the corresponding compatibility condition  holds. Then problem \eqref{eq8} has a unique solution $w$ bounded in $C^0({\Omega^{\pm}_{T}})$,and there exists some constant $C $ such that
\begin{equation}\label{1}
\|w(t)\|\leq C\left(e^{MCt}\|w^0\| + \int_0^t e^{MC(t-s)}\|f(s)\|ds\right),
\end{equation}
where $M\geq\|m\|$, $\parallel \cdot \parallel$ denote the $C^{0} $ norm.
\end{proposition}
\begin{proof}
We diagonal the system through the change variables $\widetilde{w}=R^{-1} w$. This system became as problem \eqref{eq9}. In order to solve problem \eqref{eq9} we consider the iterative scheme
\begin{equation}\label{eq13}
\begin{cases}
\partial_t\widetilde{w}^{(n+1)} +\Lambda(U,\varphi)\partial_r \widetilde{w}^{(n+1)}=
m \widetilde{w}^{(n)} +\widetilde{f},\\
[\widetilde{w_1}^{(n+1)}]=0, [\widetilde{w_3}^{(n+1)}]=0,\ \ \ \ r=b,\\
\widetilde{w}^{(n+1)}|_{t=0}=R^{-1}w^{0}(r),\\
\widetilde{w_1}^{(n+1)}= \widetilde{w_3}^{(n+1)},\ \ \  \ r=a.
\end{cases}
\end{equation}
starting with $\widetilde{w}^{(0)}|_{t=0}=R^{-1}w^{0}(r)$. In view of Proposition \ref{p1}, for each $n$ we can find a solution $\widetilde{w}^{(n+1)}$ to problem \eqref{eq13},and
$$\|\widetilde{w}^{(n+1)}(t)\| \leq C\left\{ \|\widetilde{w}^0\| + \int_0^t \|m \widetilde{w}^{(n)}(t) +\widetilde{f}(t)\|ds\right\}
\leq C\left\{ \|\widetilde{w}^0\| + \int_0^t M\|\widetilde{w}^{(n)}(t)\| +\|\widetilde{f}(t)\|ds\right\}.$$
From this  estimates it is easy to prove that the sequence$\{\widetilde{w}^{(n+1)}\}$ is a Cauchy sequences in $C^0(\Omega^{\pm}_{T})$, and its limits is solution to \eqref{eq9}. Therefore we have found a solution $w$ to \eqref{eq8}, and \eqref{1} holds.
\end{proof}

Now, we pass to  $C^1$ solutions; we consider again the problem
\begin{equation}\label{eq14}
\begin{cases}
\partial_t w + B(U,\varphi)\partial_r w = f ,\\
w_{+,1}=w_{-,1},\ w_{+,2}= w_{-,2},\  \  \ \  r=b,\\
w_{-,2}=0, \quad r=a,\\
w|_{t=0}= w^{0}(r),
\end{cases}
\end{equation}
\begin{proposition}\label{p3}
Let $ U,f,w^0$ be some families of functions and assume that they are bounded in $ C^1(\Omega^{\pm}_{T}),C^1(\Omega^{\pm}_{T}),C^1(I^{\pm})$, respectively, for some $T\in(0,T_0].$ Moreover assume that two compatibility conditions are satisfied. Then problem \eqref{eq14} has a unique solution $w$ bounded in $ C^1(\Omega^{\pm}_{T})$, satisfies \eqref{1} and
\begin{equation}\label{2}
\|\nabla w(t)\|\leq C\left\{e^{MCt}(\| \partial_r w^{0}\| +\|f(0)\|)+ \int_0^t e^{MC(t-s)}(\|f(s)\| + \|\nabla f(s)\|)ds\right\}
\end{equation}
for some constant C and M.
\end{proposition}
\begin{proof}
 we have proved $w\in C^0$,under the present assumptions we see that data $U, \tilde{f}, \tilde{w}^0 $ entering in \eqref{eq13} are continuous differentiable functions. while $m$ is barely continuous, However  in view of the particular form of $m$, we can apply the corollary \ref{cor1} and deduce that $\tilde{w}^{(n)}$ is continuously differentiable for each $n$.

the next step consists in proving that  the sequence $\{\nabla\tilde{w}^{(n)}\}$ is bounded in $L^{\infty}$ and then the sequence $\{\tilde{w}^{(n)}\}$ is equicontinuous. By induction on n that the sequence $\{\nabla\tilde{w}^{(n)}\}$ is bounded in $L^{\infty}$, using the moduli of continuity of these functions, we can get the equicontinuity of $\{\nabla\tilde{w}^{(n)}\}$.
Then Ascoli's theorem applies and existence of a $C^1$ solution $w$ to \eqref{eq14} is proved.

Then we  prove the estimates. Let $w= (w^I,w_3)^T, w^I = (w_1,w_2)^T$,
we define
$ (y,z)=(\partial_t w,\partial_r w)$, then $(y,z )$ is weak solution to
\begin{equation}\label{eq16}
\begin{cases}
\partial_t y +B\partial_r y + (\partial_t B)z =\partial_t f,\\
\partial_t z +B\partial_r z +(\partial_rB)z=\partial_r f,\\
[y^I]=0,\quad [z^I]=(B^I)^{-1} [f^I],\ \ \ \ \ r=b,\\
y_2=0,\quad z_2=\frac{1}{B_{12}}\{f_1 -y_1\} ,\quad r=a,\\
y|_{t=0}=f(0,r)-B(U^0(r),b)\partial_r w^0(r),\\
z|_{t=0} = \partial_r w^0(r),
\end{cases}
\end{equation}
where $B^I$ defined by
\[
B=
\left(
\begin{array}{cc}
B^I &0\\
0 &\lambda_2(t,r,U,\varphi)\\
\end{array}
\right).
\]
using the  Proposition 2.2.3 in \cite{C},  \eqref{2} holds.

\end{proof}
\begin{remark}
Continuing this process, we can get the piecewise $C^k $ solution of the original problem \eqref{eq7}when the initial data in the same space, and the corresponding compatibility conditions holds.

\end{remark}

 \section{The problem and asymptotic expansions of solutions}

\subsection{The problem in the Lagrangian coordinates.}
As in \cite{LW}, it is convenient to transform the system \eqref{eq3} to that in  Lagrangian coordinates, by rewriting the conservation law of mass  $\eqref{eq3}_1$ as
$$ \partial_t (r\rho^{\epsilon}) +\partial_r (r\rho^{\epsilon} u^{\epsilon})= 0,$$
we then introduce the Lagrangian coordinates $(t,x)$ with $ x=x(t,r)$ satisfying
\begin{equation}\label{eq18}
\partial_r x^{\epsilon}(t,r)= r\rho^{\epsilon}(t,r), \ \ \ \ \  \partial_t x^{\epsilon}(t,r)= -r\rho^{\epsilon}(t,r)u^{\epsilon}(t,r).
\end{equation}
We know that the coordinates transformation from $(t,r)$ to $(t,x)$, is reversible provided that $\rho^{\epsilon} > 0$. Actually, we obtain that
\begin{equation}\label{eq19}
x^{\epsilon}(t,r)=\int_a^r y\rho^{\epsilon}(t,y)dy,
\end{equation}
or
\begin{equation}
x^{\epsilon}(t,r)=\eta(r)-\int_0^t r\rho^{\epsilon}(s,r)u^{\epsilon}(s,r)ds,
\end{equation}
where\begin{equation}\label{eq20}
\eta(r):= \int_a^r y\rho_0(y)dy, \ \ \ r\in [a,\infty).
\end{equation}
Conversely, for the transformation from $(t,x)$ to $(t,r)$, we have that  $ r=r^{\epsilon}(t,x)$ satisfies
\begin{equation}\label{eq21}
\partial_x r^{\epsilon}(t,x)=\frac{1}{r\rho^{\epsilon}(t,r)} =\frac{1}{r^{\epsilon}(t,x)\rho^{\epsilon}(t,r(t,x))}, \ \ \
\partial_t r^{\epsilon}(t,x)= u^{\epsilon}(t,r)=u^{\epsilon}(t,r^{\epsilon}(t,x)).
\end{equation}
Then, we obtain that
\begin{equation}
r =r^{\epsilon}(t,x)= r_0(x)+ \int_0^t\tilde{ u}^{\epsilon}(s,x)ds,
\end{equation}
where $\tilde{u}^{\epsilon}(t,x)=u^{\epsilon}(t,r^{\epsilon}(t,x)), r_0(x):=\eta^{-1}(r)$ with $\eta^{-1}(x)$ being the inverse function of  $ \eta(r) $ given in \eqref{eq20}, since $\eta(r) $ as a function of $ r\in[a,\infty)$ is invertible provided that $\rho_0(r)>0 $.\\
Note that $\partial_x r_0(x)= \frac{1}{r_0(x)\rho_0(r_0(x))}$, which implies that
\begin{equation}\label{eq22}
r_0(x)= \sqrt{a^2+\int_0^x \frac{2}{\rho_0(r_0(y))}dy}
\end{equation}
\\
Denote by$(\tilde{\rho}^{\epsilon},\tilde{u}^{\epsilon},\tilde{v}^{\epsilon})^{T}(t,x)=(\rho^{\epsilon}, u^{\epsilon}, v^{\epsilon})^{T}(t,r^{\epsilon}(t,x)),$ and $\tau^{\epsilon}\triangleq\frac{1}{\rho^{\epsilon}}.$
For simplicity, we shall drop the tildes of notation in the following calculations.
Then
\begin{equation}\label{eq23}
\begin{cases}
\partial_t \tau^{\epsilon} - r^{\epsilon}\partial_x u^{\epsilon}
- \frac{1}{r^{\epsilon}}\tau^{\epsilon} u^{\epsilon} =0,\\

\partial_t u^{\epsilon} - r^{\epsilon}(\frac{1}{\tau^{\epsilon}})^{1+\gamma} \partial_x \tau^{\epsilon} -\frac{1}{r^{\epsilon}}(v^{\epsilon})^2
-\epsilon(\lambda +2 \mu)\{ \frac{(r^{\epsilon})^2}{\tau^{\epsilon}}\partial^2_x u^{\epsilon}
-(\frac{r^{\epsilon}}{\tau^{\epsilon}})^2\partial_x \tau^{\epsilon} \partial_x u^{\epsilon}
+2\partial_x u^{\epsilon} - \frac{\tau^{\epsilon}}{(r^{\epsilon})^2}u^{\epsilon} \}=0,\\

\partial_t v^{\epsilon} +\frac{1}{r^{\epsilon}} u^{\epsilon}v^{\epsilon}
-\epsilon \mu \{ \frac{(r^{\epsilon})^2}{\tau^{\epsilon}}\partial^2_x v^{\epsilon}
- (\frac{r^{\epsilon}}{\tau^{\epsilon}})^2 \partial_x \tau^{\epsilon} \partial_x v^{\epsilon}
+ 2\partial_x v^{\epsilon} - \frac{\tau^{\epsilon}}{(r^{\epsilon})^2}v^{\epsilon} \}=0.
\end{cases}
\end{equation}

We know that $$U^{\epsilon}(t,x)=(\tau^{\epsilon},u^{\epsilon},v^{\epsilon})^{T}(t,x)$$
satisfies the following problem in the domain $\{(t,x):t,x > 0 \}$
\begin{equation}\label{eqZ2}
\begin{cases}
\begin{split}
\mathcal{L}(U^{\epsilon}):=& \partial_t U^{\epsilon} + A(U^{\epsilon})\partial_x U^{\epsilon} + Q_1(U^{\epsilon})(U^{\epsilon},U^{\epsilon})-\epsilon B(U^{\epsilon})\partial^2_x U^{\epsilon} \\
&+ \epsilon Q_2(U^{\epsilon})(\partial_x U^{\epsilon}, \partial_x U^{\epsilon})+ \epsilon Q_3(U^{\epsilon})(U^{\epsilon},U^{\epsilon}) -\epsilon V(\partial_x U^{\epsilon}, U^{\epsilon}) =0,
\end{split}\\
U^{\epsilon}(0,x)= U^{\epsilon}_0(x)\triangleq (\tau_0,u_0,v^{\epsilon}_0)^{T} (x),\\
U^{\epsilon}_{II}(t,0)= (0, v^0(t))^{T}.
\end{cases}
\end{equation}
with $U^{\epsilon}_{II}=(u^{\epsilon}, v^{\epsilon})^{T}. $

 Where
\[
A(U^{\epsilon})=
\left(
\begin{array}{ccc}
0 & -r^{\epsilon} & 0\\
-r^{\epsilon} (\frac{1}{\tau^{\epsilon}})^{\gamma+1} &0 & 0\\
0&0& 0\\
\end{array}
\right),
\]
\\
\[
B(U^{\epsilon})= \frac{(r^{\epsilon})^2}{\tau^{\epsilon}}\cdot
diag \{0,\lambda+2\mu, \mu\},
\]

$r^{\epsilon} = r^{\epsilon}(t,x)= r_0(x)+ \int_0^t u^{\epsilon}(s,x)ds,\ \ \ r_0(x)=\sqrt{a^2+ 2\int_0^x \tau_0(y)dy}.$

$Q_{i}(\cdot,\cdot), i=1,2,3$ are quadratic forms defined as follows:

\[
\begin{split}
 &Q_1(U^{\epsilon})(U^{\epsilon},W)=\frac{1}{r^{\epsilon}}\cdot\left(-\tau^{\epsilon} w_2, -v^{\epsilon} w_3, u^{\epsilon} w_3\right)^{T},\ \ \ \ W=(w_1, w_2, w_3)^{T},\\
 &Q_2(U^{\epsilon})(\partial_x U^{\epsilon},\partial_x W)=(0, (\lambda +2\mu)(\frac{r^{\epsilon}}{\tau^{\epsilon}})^2 \partial_x \tau^{\epsilon}\partial_x w_2, \mu (\frac{r^{\epsilon}}{\tau^{\epsilon}})^2 \partial_x \tau^{\epsilon} \partial_x w_3)^{T},\\
 &Q_3(U^{\epsilon})(U^{\epsilon},W)=\frac{1}{(r^{\epsilon})^2}\cdot (0,(\lambda+2\mu)\tau^{\epsilon} w_2, \mu\tau^{\epsilon} w_3 )^{T} ,\\
 &V(\partial_xU^{\epsilon},U^{\epsilon})=2\left(0,(\lambda +2\mu)\partial_x u^{\epsilon}, \mu \partial_x v^{\epsilon}\right)^T.
\end{split}
\]

It is known that the characteristics of the system \eqref{eqZ2} are
$$ \lambda_1=  -  c , \ \ \ \lambda_2 = 0,\ \ \ \ \lambda_3=  c $$
where $ c^2= (r^{\epsilon})^2 (\frac{1}{\tau^{\epsilon}})^{\gamma+1}$. \\
We have the following property of the initial data $U^{\epsilon}_0(x)$ :
\begin{align}
\lim_{\epsilon \rightarrow 0} U^{\epsilon}_0(x)=U_0(x),\\
u_0(0)=0, v_0(0)=v^0(0).\\
v^{\epsilon}(0,x)= v_0(x) + g_1(0,\frac{x-h}{\sqrt{\epsilon}}),
\end{align}
where
\[g_1(0,\zeta)=\begin{cases}
g(0,\zeta)-v_0(h+), \quad \zeta >0,\\
g(0,\zeta)-v_0(h-), \quad \zeta <0,
\end{cases}\]
$h \triangleq \int_a^b y\rho_0(y) dy $, and $g(0,\zeta) $ be smooth function satisfies $g(0,\zeta)= v_0(h\pm),\quad \pm\zeta> 1 .$
\begin{assumption}\label{as1}
There exist positive constants $\tilde{\rho}$, and $\bar{\rho}$, such that
\begin{align}
0< \tau_0(r)\leq \frac{1}{\tilde{\rho}}\triangleq \tilde{\tau} , \ \ \ \forall \ r \geq a; \\
\tau_0(r)\rightarrow \frac{1}{\bar{\rho}}\triangleq \bar{\tau},\ \ \ \  \text{as}\  r \rightarrow \infty.
\end{align}
\end{assumption}

\subsection{ Asymptotic expansions }

As in \cite{FR}
we know that in the characteristic case, the size of the boundary layer is  $\sqrt{\epsilon}$ , so for the solutions  to the system \eqref{eqZ2}, we take the following ansatz when $\epsilon$ goes to zero:
\begin{equation}\label{l1}
U^{\epsilon}(t,x)=\sum_{j\geq0}\epsilon^{\frac{j}{2}}
U^j \left(t,x,\frac{x}{\sqrt{\epsilon}},\frac{x-h}{\sqrt{\epsilon}}\right)
= \sum_{j\geq0}\epsilon^{\frac{j}{2}} U^j \left(t,x,\xi, \zeta\right).
\end{equation}
Let's explain the idea of how to determine the profiles $ U^j(t,x,\frac{x}{\sqrt{\epsilon}},\frac{x-h}{\sqrt{\epsilon}}):= U^j (t,x,\xi, \zeta) $ for each $j$ and the detailed calculations will be given in the following subsections. We plug the \eqref{l1} into the problem \eqref{eqZ2}, and then when $\epsilon$ tends to zero and the spatial variable $ x\geq x_0 $ for a fixed $x_0 > 0$, we let $ \xi,\zeta \rightarrow +\infty$ and \eqref{eqZ2} becomes a problem depending only on $(t,x)$  whose solution denoted by $U^{I,0}(t,x).$

Next ,when $ x=O(\sqrt{\epsilon})$ , $\epsilon \rightarrow 0$, and by using the expansion
$$ U^{I,0}(t,x)=U^{I,0}(t,\sqrt{\epsilon}\xi)=
\sum_{j\geq 0}\frac{(\sqrt{\epsilon}\xi)^j}{j!}\partial^j_x U^{I,0}(t,0).$$
\eqref{eqZ2} becomes a problem depending only on $(t,\xi)$ whose solution denoted by $U^{B,0}(t,\xi)$.

When $ x-h=O(\sqrt{\epsilon})$ , $\epsilon \rightarrow 0$,
and by using the expansion
$$ U^{I,0}(t,x)=U^{I,0}(t,h+\sqrt{\epsilon}\zeta)=\sum_{j\geq 0}\frac{(\sqrt{\epsilon}\zeta)^j}{j!}\partial^j_x U^{I,0}(t,h).$$
\eqref{eqZ2} becomes a problem depending only on $(t,\zeta)$ whose solution denoted by $U^{s,0}(t,\zeta)$.
In this way the leading order profile $U^0(t,x,\xi,\zeta)$ given in \eqref{l1} in the form of
$$U^0(t,x,\xi,\zeta)= U^{I,0}(t,x)+ U^{B,0}(t,\xi)+  U^{s,0}(t,\zeta),$$
with $U^{B,0}(t,\xi), U^{s,0}(t,\zeta)$ being fast decay in $\xi,\zeta$ , respectively. \\
Then from \eqref{eqZ2} and the problems of $(U^{I,0},U^{B,0}, U^{s,0})$, we derive the problem of
$$ V^{\epsilon}(t,x)=\epsilon^{-\frac{1}{2}} \left(U^{\epsilon}(t,x)-U^0(t,x,\frac{x}{\sqrt{\epsilon}},\frac{x-h}{\sqrt{\epsilon}})\right),$$
Studying the problem of $V^{\epsilon}(t,x)$ in the same way as above, one can deduce that
\begin{equation}\label{eq26}
U^1(t,x,\xi,\zeta)= U^{I,1}(t,x)+ U^{B,1}(t,\xi)+ U^{s,1}(t,\zeta),
\end{equation}
Continuing the above process, for $j\geq 2,$ we can obtain
\begin{equation}\label{eq27}
U^j(t,x,\xi,z)= U^{I,j}(t,x)+ U^{B,j}(t,\xi)+  U^{s,j}(t,\zeta),
\end{equation}
and the problems of $ U^{I,j},U^{B,j},U^{s,j} $ and with
$ U^{B,j}(t,\xi), U^{s,j}(t,\zeta)$
being rapidly decreasing when $\xi,\zeta \rightarrow \infty$ , respectively, provided all functions involved are smoothing enough.
So we have the following form for  \eqref{l1}:
\begin{equation}\label{eq28}
U^{\epsilon}(t,x)= \sum_{j\geq0}\epsilon^{\frac{j}{2}}\left( U^{I,j}(t,x) + U^{B,j}(t,\xi) +U^{s,j}(t,\zeta)
\right).
\end{equation}

From the boundary conditions given in \eqref{eqZ2}, we deduce that $(U^{I,j},U^{B,j},U^{s,j})$ satisfy
\begin{equation}\label{B1}
\begin{cases}
U_{II}^{I,0}(t,0) + U_{II}^{B,0}(t,0)= (0,v^0(t))^{T},\\
U_{II}^{I,j}(t,0) + U_{II}^{B,j}(t,0)=0,\  j\geq 1,\\
U^{I,j}(t,h+)+U^{s,j}(t,0+)=U^{I,j}(t,h-)+U^{s,j}(t,0-),\quad j\geq 0.
\end{cases}
\end{equation}

we also have the following expansion for the coordinate transformation function $r(t,x)$ :
\begin{equation}\label{eq30}
r^{\epsilon}(t,x)= \sum_{j\geq 0}\epsilon^{\frac{j}{2}}\left( r^{I,j}(t,x)+ r^{B,j}(t,\xi)+ r^{s,j}(t,\zeta) \right),
\end{equation}
where
\begin{equation}\label{eq31}
\begin{cases}
r^{I,0}(t,x)= r_0(x) +\int_0^t u^{I,0}(s,x)ds,\\
r^{I,j}(t,x)=\int_0^t u^{I,j}(s,x)ds, \ \ j\geq 1.
\end{cases}
\end{equation}
\begin{equation}\label{eq32}
\begin{cases}
r^{B,j}(t,\xi)=&\int_0^t u^{B,j}(s,\xi)ds, \ \ j\geq 1,\\
r^{s,j}(t,\zeta)=&\int_0^t u^{s,j}(s,\zeta)ds, \ \ j\geq 1.\\
\end{cases}
\end{equation}
\subsection{derivation of problems of profiles $\{U^{I,j}(t,x), U^{B,j}(t,\xi), U^{s,j}(t,\zeta)\}_{j\geq 0}$ }
Now we derive the problems for profiles $\{U^{I,j}, U^{B,j},  U^{s,j}\}_{j\geq 0}$. Before this, we introduce notations for all $k\geq 0$,
\begin{equation}\label{eq33}
\begin{split}
U_0^{B,k}(t,\xi)=& U^{B,k}(t,\xi) +
\sum_{0\leq j\leq k} \frac{\xi^j}{j!}\partial^j_x U^{I, k-j}(t,0),\\
U_0^{s,k}(t,\zeta)=& U^{s,k}(t,\zeta) +
\sum_{0\leq j\leq k} \frac{\zeta^j}{j!}\partial^j_x U^{I, k-j}(t,h),\\
\end{split}
\end{equation}
which later will be used frequently.

Plugging \eqref{eq28} into \eqref{eqZ2}, and using above expansions, we get
\begin{equation}\label{eq35}
\mathcal{L}(U)= \sum_{k\geq -1}\epsilon^{\frac{k}{2}}\mathcal{F}_k\left( t,x,\frac{x}{\sqrt{\epsilon}}, \frac{x-h}{\sqrt{\epsilon}}\right),
\end{equation}
where
$\mathcal{F}_k( t,x,\frac{x}{\sqrt{\epsilon}},\frac{x-h}{\sqrt{\epsilon}})=
\bar{\mathcal{F}}_k( t,x) +\tilde{\mathcal{F}}_k(t,\frac{x}{\sqrt{\epsilon}}) +\tilde{\mathcal{E}}_k(t,\frac{x-h}{\sqrt{\epsilon}} )$.
By a direct calculation, we obtain each term in \eqref{eq35}. More precisely,
for $k=-1$,
\begin{align}
\bar{\mathcal{F}}_{-1}=&0,\nonumber \\
\tilde{\mathcal{F}}_{-1}=&A(U_0^{B,0})\partial_{\xi} U^{B,0},\label{eq37}
\end{align}
The terms of the expansion in \eqref{eq35} for $k=0$ are
\begin{align}
\bar{\mathcal{F}}_0=&\partial_t U^{I,0}+A(U^{I,0})\partial_x U^{I,0}+  Q_1(U^{I,0})(U^{I,0},U^{I,0}),\label{eq39}\\
\begin{split}
\tilde{\mathcal{F}}_0=&\partial_t U^{B,0}+ A(U_0^{B,0})\partial_{\xi} U^{B,1}+ dA(U_0^{B,0})\cdot U_0^{B,1}\partial_{\xi} U^{B,0}
-B(U_0^{B,0})\partial^2_{\xi} U^{B,0}\\
&- Q_2(U_0^{B,0})(\partial_{\xi} U^{B,0},\partial_{\xi} U^{B,0}) + Q_1(U_0^{B,0})(U_0^{B,0},U_0^{B,0})-\tilde{G}_0,\label{eq40}
\end{split}
\end{align}
where $\tilde{G}_0= \tilde{g}(\overline{U^{I,0}})$ for a smooth function $\tilde{g}$
 satisfying $ \tilde{g}(0)=0$.

 Similarly, for the terms in \eqref{eq35} with $k\geq 1$, we have
\begin{align}
\begin{split}
\bar{\mathcal{F}}_k=&\partial_t U^{I,k}
+ A(U^{I,0})\partial_x U^{I,k}
+ Q_1(U^{I,0})(U^{I,k},U^{I,0})
+Q_1(U^{I,0})(U^{I,0},U^{I,k}) \\
&+\bar{\Lambda}(U^{I,0})\cdot U^{I,k} + \bar{G}_k,\label{eq42}
\end{split}\\
\begin{split}
\tilde{\mathcal{F}}_k=&\partial_t U^{B,k}
+ A(U_0^{B,0})\partial_{\xi} U^{B,k+1}
-B(U_0^{B,0})\partial^2_{\xi} U^{B,k}
- Q_2(U_0^{B,0})(\partial_{\xi} U^{B,k},\partial_{\xi} U^{B,0})\\
&- Q_2(U_0^{B,0})(\partial_{\xi} U^{B,0},\partial_{\xi} U^{B,k})
+ Q_1(U_0^{B,0})(U_0^{B,k},U_0^{B,0})\\
&+ Q_1(U_0^{B,0})(U_0^{B,0},U_0^{B,k})
+\tilde{\Lambda}(U_0^{B,0})\cdot(U_0^{B,k}) +\tilde{G}_k,\label{eq43}
\end{split}
\end{align}
where
\begin{align*}
\bar{\Lambda}(U^{I,0})\cdot W =& dA(U^{I,0})\cdot W \partial_x U^{I,0}+ dQ_1(U^{I,0})\cdot W (U^{I,0},U^{I,0})  \\
\begin{split}
\tilde{\Lambda}(U_0^{B,0})\cdot W =& dA(U_0^{B,0})\cdot W \partial_{\xi} U^{B,1}
-dB(U_0^{B,0})\cdot W\partial^2_{\xi} U^{B,0}\\
&-(dQ_2(U_0^{B,0})\cdot W) (\partial_{\xi} U^{B,0},\partial_{\xi} U^{B,0})
+(dQ_1(U_0^{B,0})\cdot W) (U_0^{B,0},U_0^{B,0})
\end{split}
\end{align*}
and
$$ \bar{G}_k \ \text{is a function of } \{U^{I,j},\partial_x U^{I,j}\}_{j\leq k-1},$$
$$\tilde{G}_k\ \text{is a function of } \{ \{\overline{\partial^i_x U^{I,j}},i\leq k-j\}_{0\leq j\leq k},\{U^{B,j},\partial_{\xi} U^{B,j}\}_{j\leq k-1} \},$$
and $\tilde{\mathcal{E}}_k$ similarly with $\tilde{\mathcal{F}}_k$.

Now we derive the problem of profiles $\{U^{I,j}, U^{B,j},U^{s,j}\}_{j\geq 0}$,\ \ by discussing the equations $\bar{\mathcal{F}}_k =\tilde{\mathcal{F}}_k  =\tilde{\mathcal{E}}_k= 0$ for each $k\geq-1.$\\
\textbf{Problems of the leading order profiles}.
Discussion of the equation $\tilde{\mathcal{F}}_{-1}=0$. From the equation $\tilde{\mathcal{F}}_{-1}=0$, we obtain
\begin{equation}\label{eq48}
\begin{cases}
-r_0^{B,0}\partial_{\xi}u^{B,0}=0\\
-(\frac{1}{\tau_0^{B,0}})^{\gamma+1}r_0^{B,0}\partial_{\xi}\tau^{B,0}=0
\end{cases}
\end{equation}
which imply $(\tau^{B,0},u^{B,0})(t,\xi)\equiv 0,$
then we know that the leading term of boundary layers of the radial velocity and the pressure doesn't appear.

Similarly, from the equation $\tilde{\mathcal{E}}_{-1}=0$,
we obtain $(\tau^{s,0},u^{s,0})(t,\zeta)\equiv 0.$\\
Discussion of the equation $\tilde{\mathcal{F}}_0=0$, combining \eqref{eq48} with the boundary conditions \eqref{B1} yields
\begin{align}\label{eq49}
u^{I,0}(t,0)=0.
\end{align}
Then from$\bar{\mathcal{F}}_0=0$, it follows that $U^{I,0}(t,x)$ satisfies the following problem
\begin{equation}\label{eq50}
\begin{cases}
\partial_t U^{I,0}+ A(U^{I,0})\partial_x U^{I,0} + Q_1(U^{I,0})(U^{I,0},U^{I,0})=0, \ \ \ \ t,x > 0,x\neq h, \\
u^{I,0}(t,0)=0,\\
u^{I,0}(t,h-)=u^{I,0}(t,h+),\quad \tau^{I,0}(t,h-)= \tau^{I,0}(t,h+),
\end{cases}
\end{equation}
We endow the problem of $U^{I,0}$ with the initial data $U_0(x)$, that is,
\begin{align}\label{eq52}
U^{I,0}(0,x)=U_0(x)
\end{align}
and the zeroth compatibility condition holds.
$$\varphi(t)= b+ \int_0^t u^{I,0}(s,h)ds.$$
Thus, we know that the leading term $U^{I,0}$ of outer flow satisfies the compressible Euler equations in Lagrangian coordinates with the radial velocity vanishing on the boundary.

Discussion the equation $\tilde{\mathcal{F}}_0=0$. From the expansion of $r(t,x)$ and noting that $ u^{B,0}(t,\xi)=u_0^{B,0}(t,\xi)= 0$ we get
$$r_0^{B,0}(t,\xi)= a.$$
Then by using \eqref{eq48} and \eqref{eq50}, the equation $\tilde{\mathcal{F}}_0=0$
can be simplified as
\begin{equation}\label{eq53}
\begin{cases}
-a\partial_{\xi}u^{B,1}=0,\\
-(\frac{1}{\overline{\tau^{I,0}}})^{\gamma+1}\partial_{\xi}\tau^{B,1}
-\frac{1}{a^2}(v^{B,0}+ 2\overline{v^{I,0}})v^{B,0} = 0,\\
\partial_t v^{B,0}-a^2\mu\frac{1}{\overline{\tau^{I,0}}}\partial^2_{\xi} v^{B,0}= 0,
\end{cases}
\end{equation}
we get
\begin{align}\label{eq54}
u^{B,1}(t,\xi)\equiv 0.
\end{align}
Then the profile of $v^{B,0}(t,\xi)$:
\begin{align}\label{eq55}
\partial_t v^{B,0}=a^2\mu\frac{1}{\overline{\tau^{I,0}}}\partial^2_{\xi} v^{B,0}.
\end{align}
The boundary conditions of $v^{B,0}$ as follows:
\begin{align}\label{eq56}
v^{B,0}(t,0)=v^0(t)-\overline{v^{I,0}}(t).
\end{align}
Here we endow the problem of $v^{B,0}$ with zero initial data,i.e.,
\begin{align}\label{eq57}
v^{B,0}(0,\xi)=0,
\end{align}
so that the zeroth compatibility conditions holds:
$$ v^{B,0}(0,0)=v^0(0)-\overline{v^{I,0}}(0)=v^0(0)-v_0(0)=0.$$
Thus we get the boundary layer profile $v^{B,0}(t,\xi)$ satisfies the initial boundary value problem.
\begin{equation}\label{eq58}
\begin{cases}
\partial_t v^{B,0}=a^2\mu\frac{1}{\overline{\rho^{I,0}}}\partial^2_{\xi} v^{B,0},\quad t,\xi>0,\\
v^{B,0}(t,0)=v^0(t)-\overline{v^{I,0}}(t),\\
v^{B,0}(0,\xi)=0.
\end{cases}
\end{equation}

From the equation $\tilde{\mathcal{E}}_0 =0$
\begin{equation}\label{eq59}
\begin{cases}
r^{s,0}_0 \partial_t (\tau^{s,0}_0) -(r^{s,0}_0 )^2 \partial_{\zeta}(u^{s,1}_0) -\tau^{s,0}_0 u^{s,0}_0=0,\\
r^{s,0}_0 \partial_t (u^{s,0}_0) + \frac{(r^{s,0}_0 )^2}{(\tau^{s,0}_0)^{\gamma+1}} \partial_{\zeta}(\tau^{s,1}_0) -(v^{s,0}_0)^2=0,\\
r^{s,0}_0 \partial_t (v^{s,0}_0) -\mu \frac{(r^{s,0}_0)^3}{ \tau^{s,0}_0} \partial^2_{\zeta}(v^{s,0}_0) +  u^{s,0}_0 v^{s,0}_0 =0,
\end{cases}
\end{equation}
then from equation $\eqref{eq50}_1$ and $\eqref{eq59}_1$, 
 we obtain $u^{s,1}\equiv 0 $,
and $v^{s,0}$ satisfies
\begin{equation}\label{vs1}
\begin{cases}
r^{s,0}_0 \partial_t (v^{s,0}_0) -\mu \frac{(r^{s,0}_0)^3}{ \tau^{s,0}_0} \partial^2_{\zeta}(v^{s,0}_0) +  u^{s,0}_0 v^{s,0}_0 =0,\\
[ v^{I,0}+ v^{s,0}]=0, \quad \zeta=0,\\
v_0^{s,0}(0,\zeta)=g(0,\zeta),
\end{cases}
\end{equation}
where $g(0,\zeta)$ be smooth function and satisfies $ g(0,\zeta)= v_0(h\pm),\pm\zeta> 1$.\\

\textbf{Problems of the O($\epsilon^{\frac{1}{2}}$)-order profiles}.
We study the equation $\bar{\mathcal{F}}_1=0$. As we already have determined $u^{B,1}(t,\xi)\equiv 0$, then using \eqref{B1} we have the boundary condition of $u^{I,1}$,
\begin{align}\label{eq62}
u^{I,1}(t,0)= -u^{B,1}(t,0)=0.
\end{align}

Thus from $ \bar{\mathcal{F}}_1=0 $  we know that $U^{I,1}(t,x)$ satisfies the following linear problem in the domain
$\{(t,x):t,x>0, x\neq h\}$:
\begin{equation}\label{eq63}
\begin{cases}
\begin{split}
\partial_t U^{I,1}&
+ A(U^{I,0})\partial_x U^{I,1} + dA(U^{I,0})\cdot U^{I,1} \partial_x U^{I,0}
+ Q_1(U^{I,0})(U^{I,1},U^{I,0})\\
&+Q_1(U^{I,0})(U^{I,0},U^{I,1})
+dQ_1(U^{I,0})\cdot U^{I,1} (U^{I,0},U^{I,0})=0,
\end{split}\\
u^{I,1}(t,h-)=u^{I,1}(t,h+),\quad \tau^{I,1}(t,h-)= \tau^{I,1}(t,h+),\\
u^{I,1}(t,0)=0,\\
U^{I,1}(0,x)=0.
\end{cases}
\end{equation}

From $\tilde{\mathcal{F}}_1 =0,$ we get that $\tau^{B,1}(t,\xi),v^{B,1}(t,\xi),u^{B,2}(t,\xi)$ satisfy
\begin{equation}\label{eq61}
\begin{cases}
\partial_{\xi}\tau^{B,1}=
-\frac{({\overline{\tau^{I,0}}})^{\gamma+1}}{a^2}(v^{B,0}+ 2\overline{v^{I,0}})v^{B,0} ,\\
\partial_t \tau^{B,1} - \partial_{\xi} u^{B,2} = 0,\\
\partial_t v^{B,1} - \mu\frac{a^2}{\overline{\tau^{I,0}}}\partial^2_{\xi}v^{B,1}=f_1,\\
v^{B,1}(0,\xi)=0, \quad v^{B,1}(t,0)=-\overline{v^{I,1}}(t),
\end{cases}
\end{equation}
where
\begin{align*}
\begin{split}
f_1=-a^2 \mu \frac{1}{(\overline{\tau^{I,0}})^2}\partial_{\xi}v^{B,0}\partial_{\xi}\tau_0^{B,1}
+\mu \{\frac{2a}{\overline{\tau^{I,0}}}r_0^{B,1} -\frac{a^2}{(\overline{\tau^{I,0}})^2}\tau_0^{B,1}\}\partial^2_{\xi}v^{B,0}
+2\mu \partial_{\xi}v^{B,0} - \frac{\xi}{a}\overline{\partial_x u^{I,0}}v^{B,0},
\end{split}
\end{align*}
and $r^{B,1}_0= \xi \cdot\frac{\tau_0(0)}{a} + \int_0^t u_0^{B,1}(s,\xi)ds,$

similarly
\begin{equation}
\begin{cases}
r^{s,0}_0 \partial_t (u^{s,0}_0) - \frac{(r^{s,0}_0 )^2}{(\tau^{s,0}_0)^{\gamma+1}} \partial_{\zeta}(\tau^{s,1}_0) -(v^{s,0}_0)^2=0,\\
\begin{split}
\partial_t(\tau_0^{s,1}) - r^{s,0}_0 \partial_{\zeta}(u_0^{s,2})
-\frac{1}{r^{s,0}_0 }(\tau_0^{s,1} u^{s,0}_0 + u^{s,1}_0 \tau^{s,0}_0 )
+\frac{r^{s,1}_0 }{(r^{s,0}_0)^2 }\tau^{s,0}_0 u^{s,0}_0
- r^{s,1}_0 \partial_{\zeta}(u_0^{s,1})=0,\\
\end{split}\\
\begin{split}
\partial_t (v^{s,1}_0) &-\mu \frac{(r^{s,0}_0)^2}{\tau^{s,0}_0} \partial^2_{\zeta} (v_0^{s,1}) + \frac{1}{r^{s,0}_0}(u^{s,0}_0 v^{s,1}_0+ v^{s,0}_0 u^{s,1}_0) -\frac{r^{s,1}_0}{(r^{s,0}_0)^2}v^{s,0}_0 u^{s,0}_0\\
&=\mu \{2\frac{ r^{s,0}_0}{\tau^{s,0}_0} r^{s,1}_0 - (\frac{r_0^{s,0}}{\tau_0^{s,0}})^2 \tau^{s,1}_0 \} \partial^2_{\zeta} (v^{s,0}_0)
+ \mu(\frac{r^{s,0}_0}{\tau_0^{s,0}})^2 \partial_{\zeta} (v^{s,0}_0) \partial_{\zeta}(\tau_0^{s,1})+ 2\mu\partial_{\zeta} (v_0^{s.0}),
\end{split}\\
v^{s,1}(0,\zeta)=0,
\quad v^{s,1}(t,0+)-v^{s,1}(t,0-)= v^{I,1}(t,h-)-v^{I,1}(t,h+),
\end{cases}
\end{equation}

\textbf{Problems of the other order profiles}.
Continuing above process to study the equations $\bar{\mathcal{F}}_k,\tilde{\mathcal{E}}_k,\tilde{\mathcal{F}}_k =0, k\geq 2$.
For all $k\geq 2$, $U^{I,k}(t,x)$ are solutions to linearized Euler equations similar to \eqref{eq50} with the boundary condition $u^{I,k}|_{x=0}=-u^{B,k}(t,0)$, and $U^{B,k}(t,\xi),U^{s,k}(t,\zeta)$ satisfy the linearized parabolic problems similar to that $U^{B,1},U^{s,1}$, respectively.
\\
\subsection{Conclusion}
As mentioned above, we conclude are follows:
\\
\textbf{Conclusion}
The solution $ U^{\epsilon} =(\tau^{\epsilon},u^{\epsilon},v^{\epsilon})^{T} $ to the problem \eqref{eqZ2} formally admits the following asymptotic expansion:
\begin{align}\label{eq68}
U^{\epsilon}(t,x)= \sum_{j\geq 0}\epsilon^{\frac{j}{2}}\left( U^{I,j}(t,x)+ U^{B,j}(t,\frac{x}{\sqrt{\epsilon}})+U^{s,j}(t,\frac{x-h}{\sqrt{\epsilon}})\right)
\end{align}
for rapidly decaying $\{U^{B,j}(t,\xi),U^{s,j}(t,\zeta)\}_{j\geq 0}$ in $\xi ,\zeta \rightarrow\infty$,respectively, where for $j\geq 0$,
$$U^{I,j}=(\tau^{I,j},u^{I,j},v^{I,j})^{T}, \ \ U^{B,j}(t,\xi)=(\tau^{B,j},u^{B,j},v^{B,j})^{T},\ \ U^{s,j}=(\tau^{s,j},u^{s,j},v^{s,j})^{T}.$$
And we have the following properties:
\begin{enumerate}
\item[(i)] $U^{I,0}(t,x)$ satisfies the following initial boundary value problem for the compressible Euler equations in $\{(t,x): t>0,x>0,x\neq h\}$:
    \begin{equation}\label{eq69}
  \begin{cases}
  \partial_t U^{I,0}+ A(U^{I,0})\partial_x U^{I,0} + Q_1(U^{I,0})(U^{I,0},U^{I,0})=0,  \\
  u^{I,0}(t,h-)=u^{I,0}(t,h+),\quad \tau^{I,0}(t,h-)= \tau^{I,0}(t,h+),\\
  u^{I,0}(t,0)=0,\\
  U^{I,0}(0,x)=U_0(x);\\
  \end{cases}
  \end{equation}
  and $U^{I,1}(t,x)$ satisfies the following problem for the linearized compressible Euler equations in $\{(t,x): t>0,x>0,x\neq h\}$:
  \begin{equation}\label{eq70}
  \begin{cases}
  \begin{split}
  \partial_t U^{I,1}&
  + A(U^{I,0})\partial_x U^{I,1} + dA(U^{I,0})\cdot U^{I,1} \partial_x U^{I,0}
  + Q_1(U^{I,0})(U^{I,1},U^{I,0})\\
  &+Q_1(U^{I,0})(U^{I,0},U^{I,1})
  +dQ_1(U^{I,0})\cdot U^{I,1} (U^{I,0},U^{I,0})=0,
  \end{split}\\
  u^{I,1}(t,h-)=u^{I,1}(t,h+),\quad \tau^{I,1}(t,h-)= \tau^{I,1}(t,h+)+\phi^{1}(t),\\
  u^{I,1}(t,0)=-u^{B,1}(t,0)=0,\\
  U^{I,1}(0,x)=0;\\
  \end{cases}
  \end{equation}
For all $j\geq 2,U^{I,j}(t,x)$ are solutions to the linear problems similar to \eqref{eq70} with the boundary condition $u^{I,j}(t,0)=-u^{B,j}(t,0)$;

\item[(ii)] The leading boundary layer profiles $U^{B,0}(t,\xi)=(\tau^{B,0} ,u^{B,0}, v^{B,0})^{T}(t,\xi)$ satisfies that
  \begin{align}\label{eq71}
  (\tau^{B,0},  u^{B,0})(t,\xi)\equiv 0,
  \end{align}
  and the following boundary value problem of nonlinear parabolic equations:
  \begin{equation}\label{eq72}
  \begin{cases}
  \partial_t v^{B,0}=\frac{a^2\mu}{\overline{\tau^{I,0}}}\partial^2_{\xi} v^{B,0}, \ \quad \  t,\  \xi>0,\\
  v^{B,0}(t,0)=v^0(t)-\overline{v^{I,0}}(t),\\
  v^{B,0}(0,\xi)=0;
  \end{cases}
  \end{equation}
  For the next order profile $U^{B,1}(t,\xi)=(\tau^{B,1}, u^{B,1}, v^{B,1})^{T}(t,\xi)$.\\
  $ u^{B,1}(t,\xi)\equiv0 $  and $ v^{B,1}(t,\xi)$ satisfies the following linearized problem of \eqref{eq72}:
  \begin{align}
  \begin{cases}
  \begin{split}
  &\partial_t v^{B,1} - a^2\mu\frac{1}{\overline{\tau^{I,0}}}\partial^2_{\xi}v^{B,1}=f_1,
  \end{split}\\
  &v^{B,1}(t,0)=-\overline{v^{I,1}}(t),\\
  &v^{B,1}(0,\xi)=0,
  \end{cases}
  \end{align}
  where $f_1$ given as before,  and $\tau^{B,1}(t,\xi)$ is given by
  $$\partial_{\xi}\tau^{B,1}=
-\frac{({\overline{\tau^{I,0}}})^{\gamma+1}}{a^2}(v^{B,0}+ 2\overline{v^{I,0}})v^{B,0} .$$
  For each $j\geq 2$, the profiles $ U^{B,j}(t,\xi)=(\tau^{B,j},u^{B,j},v^{B,j})^{T}(t,\xi)$  satisfy linear problems similar to $U^{B,1}(t,\xi)$;
\item[(iii)] the leading vortex layer profile $U^{s,0}(t,\zeta)=(\tau^{s,0} ,u^{s,0}, v^{s,0})^{T}(t,\zeta)$ satisfies that
  $$(\tau^{s,0},  u^{s,0})(t,\zeta)\equiv 0,$$
  and the following boundary value problem of nonlinear parabolic equations:
  \begin{equation}
  \begin{cases}
  \partial_t (v_0^{s,0}) - \mu \frac{(r_0^{s,0})^2}{\tau_0^{s,0}}\partial^2_{\zeta}(v_0^{s,0})
  + \frac{1}{r_0^{s,0}}u_0^{s,0} v_0^{s,0}=0,\quad t>0, \zeta\neq 0,\\
  v^{s,0}(t,0+)-v^{s,0}(t,0-)= v^{I,0}(t,h+)-v^{I,0}(t,h-),\\
  v_0^{s,0}(0,\zeta)= g(0,\zeta),
  \end{cases}
  \end{equation}
  where $g(0,\zeta)$ be smooth function and satisfies $g(0,\zeta)= v^{I,0}(0,h\pm), |\zeta|>1.$

  For the next order profile $U^{s,1}(t,\zeta)=(\tau^{s,1}, u^{s,1}, v^{s,1})^{T}(t,\zeta)$.
  $ u^{s,1}(t,\zeta)\equiv0 $  and $ v^{s,1}(t,\zeta)$ satisfies the following linearized problem:
  \begin{equation}\label{eqs1}
  \begin{cases}
  \begin{split}
  \partial_t (v^{s,1}_0) &-\mu \frac{(r^{s,0}_0)^2}{\tau^{s,0}_0} \partial^2_{\zeta} (v_0^{s,1}) + \frac{1}{r^{s,0}_0}(u^{s,0}_0 v^{s,1}_0+ v^{s,0}_0 u^{s,1}_0) -\frac{r^{s,1}_0}{(r^{s,0}_0)^2}v^{s,0}_0 u^{s,0}_0=f_3
  \end{split}\\
  v^{s,1}(t,0+)-v^{s,1}(t,0-)= v^{I,1}(t,h-)-v^{I,1}(t,h+),\\
  v^{s,1}(0,\zeta)=0,\\
  \end{cases}
  \end{equation}
  where
  \begin{align*}
   f_3= \mu \{2\frac{ r^{s,0}_0}{\tau^{s,0}_0} r^{s,1}_0 - (\frac{r_0^{s,0}}{\tau_0^{s,0}})^2 \tau^{s,1}_0 \} \partial^2_{\zeta} (v^{s,0}_0)
  + \mu(\frac{r^{s,0}_0}{\tau_0^{s,0}})^2 \partial_{\zeta} (v^{s,0}_0) \partial_{\zeta}(\tau_0^{s,1})+ 2\mu\partial_{\zeta} (v_0^{s.0}),
  \end{align*}
    and $\tau^{s,1}(t,\zeta)$ given by
  $$r^{s,0}_0 \partial_t (u^{s,0}_0) - \frac{(r^{s,0}_0 )^2}{(\tau^{s,0}_0)^{\gamma+1}} \partial_{\zeta}(\tau^{s,1}_0) -(v^{s,0}_0)^2=0.$$
  For each $j\geq 2$, the profiles $ U^{s,j}(t,\zeta)=(\tau^{s,j},u^{s,j},v^{s,j})^{T}(t,\zeta)$ satisfy linear problems similar to $U^{s,1}(t,\zeta)$.
\end{enumerate}

\section{study on the profiles $\{U^{I,j},U^{B,j},U^{s,j}\}_{j\geq 0}$}In this section, we study the well-posedness of problem of profiles $\{U^{I,j},U^{B,j},U^{s,j}\}_{j\geq 0}$ derived in section 3.

\subsection{well-posedness of the problem of $U^{I,0}$}
From the Conclusion, we know that that the leading term $U^{I,0}(t,x)= (\tau^{I,0} ,u^{I,0}, v^{I,0})^{T}(t,x)$ of outer flow satisfies the following nonlinear initial-boundary value problem in the domain $\{(t,x)| t,x > 0, x\neq h\}$:
\begin{equation}\label{eq73}
\begin{cases}
\partial_t U^{I,0}+ A(U^{I,0})\partial_x U^{I,0} + Q_1(U^{I,0})(U^{I,0},U^{I,0})=0,\\
u^{I,0}(t,h-)=u^{I,0}(t,h+),\quad \tau^{I,0}(t,h-)= \tau^{I,0}(t,h+),\\
u^{I,0}(t,0)=0,\\
U^{I,0}(0,x)=U_0(x);
\end{cases}
\end{equation}

One can see that the equations in \eqref{eq73} are the compressible Euler equations in the Lagrangian coordinates.

To study the problem \eqref{eq73}, let's first consider the following nonlinear problem of $U(t,r)= (\tau, u, v)^{T}(t,r)$ in the Eulerian coordinates:
\begin{equation}\label{eq74}
\begin{cases}
\partial_t U + \tilde{A}(U)\partial_r U + \tilde{Q}_1(U,U)=0, \ \ \ \ t>0, r>a,r\neq \varphi(t),\\
[\tau]=[u]=0,   \quad \ \ r= \varphi(t),\\
u=0,\quad r=a,\\
U|_{t=0}:= U_0(r)= (\tau_0, u_0, v_0)^{T}(r),\\
\end{cases}
\end{equation}
where \\
\[
\tilde{A}(U)=
\left(
\begin{array}{ccc}
u & -\tau & 0\\
\\
-\tau^{-\gamma} &u & 0\\
\\
0&0& u\\
\end{array}
\right),
\]
$$ \tilde{Q}_1(U,U)=\frac{1}{r}\left(-\tau u, -v^2, uv\right)^{T}$$
It is not difficult to show the boundary conditions given in \eqref{eq74} is sufficient to solve this problem.

Thus,provided that the initial data $U_0(r)$ satisfies that
$$0<\tau_0(r) \leq \tilde{\tau},\ \ \ \ \forall r\geq a,$$
with positive constant $\tilde{\tau}, \bar{\tau}$ given as before,
$$U_{0,\pm}(r)-(\bar{\tau},0,0)^{T}\in H^{s+2}, \ \ \ \ s>\frac{3}{2},$$
and the compatibility conditions of \eqref{eq74} holding up to order $s$, we know that $U_{0,\pm}(r)\in C^s$,then from the section 2 we know the solution $U_{\pm}(t,r)\in C^s$. So we know the value of solution on free boundary, then straighten the free boundary, and using the symmetric hyperbolic theroy, there exists a $T_0>0 $ and a unique solution $U(t,r)$ to equation \eqref{eq74} such that,
$$ U_{\pm}-(\bar{\tau},0,0)^{T} \in\bigcap_{0\leq k\leq s} W^{k,\infty}(0,T_0;H^{s-k})$$
and $ \tau(t,r)>0.$ One can refer to \cite{AM} or\cite{MET}, for instance, for the proof of this result.

Now, we introduce the lagrangian coordinates $(t,x)$ in the problem \eqref{eq74} by the relation
\begin{align}\label{eq75}
r^{\epsilon}= r^{\epsilon}(t,x)= r_0(x) + \int_0^t \tilde{u}^{\epsilon}(s,x)ds
\end{align}
where $\tilde{u}^{\epsilon}(s,x)= u^{\epsilon}(t,r^{\epsilon}(t,x))$ and $ r_0(x)=\eta^{-1}(r) $ is the inverse function of
\begin{align}\label{eq76}
\eta(r):= \int_a^r y\rho_0(y)dy,  \ \ \  r\geq a.
\end{align}
By using the argument similar to that given at beginning of  section 3, we know that the transformation from $(t,r)$ to $(t,x)$ is reversible and $U(t,r(t,x))$ is a solution to \eqref{eq73}. Moreover, by combining \eqref{eq31} with \eqref{eq75},it follows that
\begin{equation}\label{eq77}
\begin{split}
&r(t,x)= r^{I,0}(t,x)= r_0(x) + \int_0^t u^{I,0}(s,x)ds,\\
&r_0(x)=\sqrt{a^2 +2 \int_0^x \tau_0(y)dy},\\
\\
&\partial_x r(t,x)=\partial_x r^{I,0}(t,x)= \frac{\tau^{I,0}(t,x)}{r^{I,0}(t,x)}.
\end{split}
\end{equation}
So we have the following result
\begin{proposition}\label{p4}
Let the initial data $U_0(x)= (\tau_0, u_0, v_0)^{T}(x) $ satisfy Assumption, and
$U_{0,\pm}(r)-(\bar{\tau},0,0)^{T}\in H^{s+2},$ with $\bar{\tau}$ being positive constant given before, and $s >\frac{3}{2}$ being an integer. Also, the compatibility conditions of \eqref{eq73} hold up to order $s$. Then there exists a $ T_0>0 $ and a unique solution $U^{I,0}(t,x)$ to \eqref{eq73} such that
\begin{align}\label{eq78}
U^{I,0}_{\pm}-(\bar{\tau},0,0)^{T} \in\bigcap_{0\leq k\leq s} W^{k,\infty}(0,T_0;H^{s-k})
\end{align}
Moreover, we have $\tau^{I,0}(t,x)> 0$.
\end{proposition}
\subsection{study on the profile $U^{B,0}$}
As in Conclusion,the key point of determining the profile $U^{B,0}$ is to study the equation \eqref{eq72}.

The idea is similarly with \cite{LW}, so we sketch the idea
We first construct an auxiliary function to homogenize the boundary conditions. More precisely,
let$ f(t,\xi)$ be smooth and satisfies $f(0,\xi)=0, \xi> 0,$
$$f(t,0)= v^0(t)-\overline{v^{I,0}}(t) \quad t>0,$$
and $ f(t,\xi)=0, \quad t>0, \quad \xi\geq 1$.\\
Setting $$v(t,\xi)= v^{B,0}(t,\xi)-f(t,\xi),$$
\\
we know that $v(t,\xi)$ satisfies the following problem in $\{ t>0,\xi>0\}:$
\begin{equation}\label{ws1}
\begin{cases}
\partial_t v-a^2 \mu \frac{1}{\overline{\tau^{I,0}}}\partial^2_{\xi}(v+f) =-\partial_t f \triangleq \tilde{g},\\
v(t,0)=0,\\
v(0,\xi)=0,
\end{cases}
\end{equation}
since $ \overline{v^{I,0}}(t)\in H^{s-1} $, we know $\tilde{g}\in H^{s-2}((0,T_0)\times R^+) .$
By using the results of $U^{I,0}$ given in proposition \ref{p4}, we know that there exists a positive constant $M_1$ such that
$$M_1^{-1} \leq \frac{a^2}{ \overline{\tau^{I,0}}}\leq M_1,\quad\partial_t( \frac{a^2} {\overline{\tau^{I,0}}})\leq M_1, \forall t \in[0,T_0],$$
and $\tilde{g}\in H^{s-2}(0,T_0)\times R^+).$ Moreover, for a positive constant $0<M_0 < \frac{M_1}{2},$ and for all $n \in N$,
$$ \|\langle\xi\rangle^n f\|_{W^{1,\infty}(0,T_0;H^2)}+ \|\langle\xi\rangle^n \tilde{g}\|^2_{L^{\infty}(0,T_0 ;L^2)} \leq M_0^2,$$
with $\langle\xi\rangle\triangleq \sqrt{1+ \xi^2}.$
\\
The main result of this section is as follows.
\begin{proposition}\label{p5}
For the problem \eqref{ws1}, there exists a $T_1: 0<T_1 \leq T_0$ and a unique solution $v(t,\xi)$ to \eqref{ws1} such that
$$ v\in W^{1,\infty}(0,T_1;H^1)\cap H^1(0,T_1;H^2),$$
and for all $ n \in N$ the estimate
\begin{equation}
\|\langle\xi\rangle^n v\|_{L^{\infty}(0,T_1;H^1)} +\|\langle\xi\rangle^n v\|_{L^2 (0,T_1;H^2)}
\leq M (\|\langle\xi\rangle^n f\|_{L^2 (0,T_1;H^2)} + \|\langle\xi\rangle^n \tilde{g}\|_{L^2 (0,T_1;L^2)} )
\end{equation}
 holds for a positive constant $M= M(T_1,n,M_1)$ .  Moreover if $ \tilde{g}\in H^{m}((0,T_1)\times R^+)$
and
$$\partial^k_t \tilde{g}(0,\xi)=0,\quad k=0,1,\ldots,m-1, \ \ \xi  \geq 0,$$
we get
$$ \langle\xi\rangle^n v(t,\xi)\in \bigcap_{k+[\frac{l}{2}]\leq m} (W^{k,\infty}(0,T_1;H^l)\cap H^k(0,T_1;H^{l+1})). $$
\end{proposition}
\begin{proof}
(1) we are going to estimate the weighted $L^2 -$estimate of $v(t,\xi)$. Multiplying \eqref{ws1} by $\langle\xi\rangle^{2n} v, n \in N$
 and integrating over $\mathbf{R^+}$ with respect to $\xi$,we get
\begin{equation}\label{wes1}
\frac{1}{2}\frac{d}{dt}\|\langle\xi\rangle^n v\|^2
 + a^2\mu \frac{1}{\overline{\tau^{I,0}}} \int_0^{\infty} \partial_{\xi}(v+f) \cdot \partial_{\xi}(\langle\xi\rangle^{2n} v)d\xi
=\int_0^{\infty} \tilde{g}\cdot \langle\xi\rangle^{2n} v d\xi
\end{equation}
since
\begin{equation*}
\begin{split}
a^2\mu \frac{1}{\overline{\tau^{I,0}}}& \int_0^{\infty} \partial_{\xi}(v+f) \cdot \partial_{\xi}(\langle\xi\rangle^{2n} v)d\xi\\
&\geq \frac{\mu}{M_1}\|\langle\xi\rangle^n \partial_{\xi} v\|^2
-\mu M_1 ( ( \|\langle\xi\rangle^n v\|+\|\langle\xi\rangle^n f\|)\cdot \|\langle\xi\rangle^n \partial_{\xi} v\|
 +2n  \|\langle\xi\rangle^n v\|\cdot \|\langle\xi\rangle^n \partial_{\xi} v\|)\\
&\geq \frac{\mu}{2M_1}\|\langle\xi\rangle^n \partial_{\xi} v\|^2 -C (\|\partial_{\xi} f\|^2 + \|\langle\xi\rangle^n v\|^2)
\end{split}
\end{equation*}
using the Gronwall inequality there exists a positive constant$ C=C(T_1,n ,M_1)$ such that
\begin{equation}
\|\langle\xi\rangle^n v\|^2_{L^{\infty}(0,T_1;L^2)} + \|\langle\xi\rangle^n \partial_{\xi} v\|^2_{L^2 (0,T_1;L^2)}
\leq C (\|\langle\xi\rangle^n f\|^2_{L^2 (0,T_1;H^1)}+\|\langle\xi\rangle^n \tilde{g}\|^2_{L^2 (0,T_1;L^2)})
\end{equation}
(2) we want to get the weight estimate of $\partial_{\xi} v$. Multiplying \eqref{ws1} by $-\langle\xi\rangle^{2n} \partial^2_{\xi} v$
and integrating over $\mathbf{R^+}$ with respect to $\xi$, it follows by integration by parts that
\begin{equation}\label{wes2}
\begin{split}
\frac{1}{2}\frac{d}{dt}\|\langle\xi\rangle^n \partial_{\xi} v\|^2
&+ \int_0^{\infty} 2n\xi \langle\xi\rangle^{2n-2}\partial_{\xi}v \cdot \partial_t v d\xi \\
&+ a^2\mu \frac{1}{\overline{\tau^{I,0}}} \int_0^{\infty} \langle\xi\rangle^{2n} \partial^2_{\xi} (v+f) \cdot \partial^2_{\xi} v d\xi
=- \int_0^{\infty} \langle\xi\rangle^{2n}\tilde{g}\cdot \partial^2_{\xi} v d\xi.
\end{split}
\end{equation}
using
$$\left|  \int_0^{\infty} 2n\xi \langle\xi\rangle^{2n-2}\partial_{\xi}v \cdot \partial_t v d\xi \right|
\leq 2n\|\langle\xi\rangle^n \partial_{\xi} v\| \cdot \|\langle\xi\rangle^n \partial_t v\|,$$
and
$$ \|\langle\xi\rangle^n \partial_t v\| \leq C(\|\langle\xi\rangle^n \partial^2_{\xi} (v+f)\|+ \|\langle\xi\rangle^n \tilde{g}\|).$$
Thus we have
\begin{equation*}
\left|\int_0^{\infty} 2n\xi \langle\xi\rangle^{2n-2}\partial_{\xi}v \cdot \partial_t v d\xi \right|
\leq \frac{\mu}{12M_1}\|\langle\xi\rangle^n \partial^2_{\xi} v\|^2
+ C(\|\langle\xi\rangle^n \partial_{\xi} v\|^2+\|\langle\xi\rangle^n \partial^2_{\xi} f\|^2 +\|\langle\xi\rangle^n \tilde{g}|^2).
\end{equation*}
since
\begin{align*}
 a^2\mu \frac{1}{\overline{\tau^{I,0}}} \int_0^{\infty} \langle\xi\rangle^{2n} \partial^2_{\xi} (v+f) \cdot \partial^2_{\xi} v d\xi
\geq\frac{\mu}{2M_1} \|\langle\xi\rangle^n \partial^2_{\xi} v\|^2 -C\|\langle\xi\rangle^n \partial^2_{\xi} f\|^2,
\end{align*}
and
$$ \left|\int_0^{\infty} \langle\xi\rangle^{2n}\tilde{g}\cdot \partial^2_{\xi} v d\xi \right|
\leq \frac{\mu}{12M_1}\|\langle\xi\rangle^n \partial^2_{\xi} v\|^2 + C \|\langle\xi\rangle^n \tilde{g}\|^2.
$$
then using Gronwall inequality there exists a positive constant $M=M(M_1,n,T_1)$, such that
$$
\|\langle\xi\rangle^n v\|^2_{L^{\infty}(0,T_1;L^2)} + \|\langle\xi\rangle^n \partial_{\xi} v\|^2_{L^2 (0,T_1;L^2)}
\leq M \|\langle\xi\rangle^n \tilde{g}\|^2_{L^2 (0,T_1;L^2)}.
$$
(3)When $\tilde{g}\in H^m(0,T_1)$ and $$\tilde{g}^{(k)}(0)=0, \quad k= 0,1,\ldots,m-1$$
applying the operator $\partial_t $ to the problem \eqref{ws1} yields
\begin{equation}\label{ws2}
\begin{cases}
\partial^2_t v -a^2\mu\left(\frac{1}{\overline{\tau^{I,0}}} \partial_t (\partial^2_{\xi}(v+f)) +\partial_t \frac{1}{\overline{\tau^{I,0}}}\partial^2_{\xi}(v+f) \right )=\partial_t \tilde{g}\\
\partial_t v(t,0)=0, \quad \partial_t v(0,\xi)=0,
\end{cases}
\end{equation}
Though the same argument as above, there exist a positive constant $M'=M'(T_1,n,M_1)$, such that
$$\|\langle\xi\rangle^n \partial_t v\|_{L^{\infty}(0,T_1;H^1)} +\|\langle\xi\rangle^n \partial_t v\|_{L^2 (0,T_1;H^2)}
\leq M' (\|\langle\xi\rangle^n f\|_{H^1 (0,T_1;H^2)}+\|\langle\xi\rangle^n \tilde{g}\|_{H^1 (0,T_1;L^2)} ),
$$
Moreover it implies $\langle\xi\rangle^n \partial^2_t v \in L^2 (0,T_1;L^2)$ .\\
Next by applying the operator $\partial^j_t ,j\leq m $ to the equation \eqref{ws1} and using arguments similarly to those above, we can get
$$\langle\xi\rangle^n v(t,\xi) \in W^{m,\infty}(0,T_1;H^1)\cap H^m(0,T_1;H^2)\cap H^{m+1}(0,T_1;L^2)$$
and the corresponding estimates of the solution $v(t,\xi)$ in these spaces.\\
(4) recall from \eqref{ws1} that
\begin{equation}\label{ws3}
\partial^2_{\xi} v =-\partial^2_{\xi} f + \frac{\overline{\tau^{I,0}}}{a^2\mu }(\partial_t v- \tilde{g})
\end{equation}
using the results of the third step , we know that
\begin{equation}\label{wes3}
\langle\xi\rangle^n \partial^2_{\xi} v(t,\xi) \in W^{m-1,\infty}(0,T_1;L^2)\cap H^{m-1}(0,T_1;H^1)\cap H^{m}(0,T_1;L^2).
\end{equation}
Then applying the operator $\partial_{\xi}$ to \eqref{ws3} and combining \eqref{wes3} yields
$$\langle\xi\rangle^n \partial^3_{\xi} v(t,\xi) \in W^{m-1,\infty}(0,T_1;L^2)\cap H^{m-1}(0,T_1;H^1)$$
Continuing this process, we finally can get
$$ \langle\xi\rangle^n v(t,\xi)\in \bigcap_{k+[\frac{l}{2}]\leq m} (W^{k,\infty}(0,T_1;H^l)\cap H^k(0,T_1;H^{l+1}))$$
\end{proof}
Hence we obtain similar results for the original problem \eqref{eq72} of $v^{B,0}(t,\xi)$ immediately by using Proposition \ref{p5}. Indeed combining Proposition \ref{p4} we conclude the following.
\begin{proposition}
Under the assumption of Proposition\ref{p4}, and for the parameters $T_0$ and s given in Proposition \ref{p4}, we choose the initial data $U_0$ of problem \eqref{eq73} such that the compatibility conditions of \eqref{eq72} hold up to order $s-3$. Then for the problem \eqref{eq72}, there exists a $T_1: 0<T_1\leq T_0$ and a unique solution $v^{B,0}(t,\xi)$ to \eqref{eq72} such that for all $n\in N$,
$$\langle\xi\rangle^n v^{B,0}(t,\xi)\in \bigcap_{k+[\frac{l}{2}]\leq s-2} (W^{k,\infty}(0,T_1;H^l)\cap H^k(0,T_1;H^{l+1}))$$
\end{proposition}
\subsection{study on the profiles $\{(U^{I,j},U^{B,j})\}_{j\geq 1}$}

As showing in conclusion the next profile $U^{I,1}(t,x)$ of out flow satisfies the following linear problem in $\{(t,x)| t,x >0,x\neq h\}$:
\begin{equation}\label{eq80}
  \begin{cases}
  \begin{split}
  \partial_t U^{I,1}&
  + A(U^{I,0})\partial_x U^{I,1} + dA(U^{I,0})\cdot U^{I,1} \partial_x U^{I,0}
  + Q_1(U^{I,0})(U^{I,1},U^{I,0})\\
  &+Q_1(U^{I,0})(U^{I,0},U^{I,1})
  +dQ_1(U^{I,0})\cdot U^{I,1} (U^{I,0},U^{I,0})=0,
  \end{split}\\
  u^{I,1}(t,h-)=u^{I,1}(t,h+),\quad \tau^{I,1}(t,h-)= \tau^{I,1}(t,h+)+\phi^1(t),\\
  u^{I,1}(t,0)=-u^{B,1}(t,0)=0,\\
  U^{I,1}(0,x)=0.
  \end{cases}
\end{equation}
First, we observe that \eqref{eq80} is a symmetrizable  hyperbolic system.  Indeed by letting
$$ S(U^{I,0})=\frac{1}{r^{I,0}}diag\{(\tau^{I,0})^{-(\gamma+1)},1,1\}$$
we have
\[
S(U^{I,0})\cdot A(U^{I,0})=
\left(
\begin{array}{ccc}
0 & -(\frac{1}{\tau^{I,0}})^{\gamma+1} &0\\

-(\frac{1}{\tau^{I,0}})^{\gamma+1} & 0 &0\\

0 & 0 &0
\end{array}
\right)
\]
Then by applying the classical theory of symmetrizable hyperbolic system (cf, \cite{AM},\cite{MET}) for the problem \eqref{eq80}, and using boundary data $ u^{I,0}(t,0) \in H^s(0,T_1)$.
there exists a unique solution $U^{I,1}(t,x)$ to \eqref{eq80} such that
\begin{align}\label{eq81}
U^{I,1}_{\pm} \in L^{\infty}(0,T_1;H^{s})\cap Lip (0,T_1;H^{s-1})
\end{align}
Moreover by using the the equation given in \eqref{eq80} we get that
\begin{align}\label{eq82}
U^{I,1}_{\pm}\in\bigcap_{k=0}^{s} W^{k,\infty}(0,T_1;H^{s-k})
\end{align}
Here we also need to the choose the initial data $U_0$ of the problem \eqref{eq78}, such that the compatibility conditions of \eqref{eq73} hold up to order $s-1$.
\\
\textbf{For the profile $U^{B,1}(t,\xi)$}, from the Conclusion, we know that $v^{B,1}(t,\xi)$ satisfy that the following linear problem in $\{(t,\xi): t,\xi >0\}:$
\begin{align}\label{ws4}
  \begin{cases}
  \begin{split}
  &\partial_t v^{B,1} - a^2\mu\frac{1}{\overline{\tau^{I,0}}}\partial^2_{\xi}v^{B,1}=f_1,
  \end{split}\\
  &v^{B,1}(t,0)=-\overline{v^{I,1}}(t),\\
  &v^{B,1}(0,\xi)=0,
  \end{cases}
\end{align}
where $f_1$ given as before.

The problem \eqref{ws4} is a classical linear parabolic type for $v^{B,1}(t,\xi)$. So by using the argument similar to that given in subsection 4.2, we can obtain the following weight estimates for $(\tau^{B,1},v^{B,1})(t,\xi):$
$$\langle\xi\rangle^n (\tau^{B,1},v^{B,1})(t,\xi) \ \text{is bounded in }
\bigcap_{k+[\frac{l}{2}]\leq s-2}\left(W^{k,\infty}(0,T_1;H^l)\cap H^k(0,T_1;H^{l+1}) \right),$$
which is immediately implies the boundedness of $u^{B,2}(t,\xi)$ from \eqref{eq61},
$$ \langle\xi\rangle^n u^{B,2} \ \text{is bounded in }
\bigcap_{k+[\frac{l+1}{2}]\leq s-2}\left(W^{k,\infty}(0,T_1;H^l)\cap H^k(0,T_1;H^{l+1}) \right).$$
\\

Continuing this process for $1\leq j\leq \frac{s+1}{3}$, we finally obtain that under the assumption of Proposition \ref{p4},
$$ U^{I,j}_{\pm} \in \bigcap_{k=0}^{s-3(j-1)} W^{k,\infty}(0,T_1;H^{s+3-3j-k}),$$
and
$$\langle\xi\rangle^n (\tau^{B,j},v^{B,j})(t,\xi) \in \bigcap_{k+[\frac{l}{2}]\leq s+1-3j}\left(W^{k,\infty}(0,T_1;H^l)\cap H^k(0,T_1;H^{l+1}) \right),$$
$$\langle\xi\rangle^n u^{B,j+1} \in \bigcap_{k+[\frac{l+1}{2}]\leq s+1-3j}\left(W^{k,\infty}(0,T_1;H^l)\cap H^k(0,T_1;H^{l+1}) \right),$$
for all $1\leq j\leq \frac{s+1}{3}$. Here the  initial data $U_0$ of the problem \eqref{eq73} is chosen such that the compatibility conditions of the corresponding problems holds.
\subsection{study on the profile $\{U^{s,j}\}_{j\geq 0}$}

We first construct an auxiliary function . More precisely,
let $ f_1(t,\zeta)$ be smoothness and satisfies
\[f_1(t,\zeta)=\begin{cases}
v^{I,0}(t,h+),\quad \zeta >1,\\
\\
v^{I,0}(t,h-),\quad \zeta <-1,
\end{cases}
\]
and $f_1(0,\zeta)= g(0,\zeta)$.

So we get
 $$\partial_t f_1(t,\zeta) +\frac{\varphi'(t)}{\varphi(t)} f_1(t,\zeta)= 0 , |\zeta|> 1.$$
 Setting
$$w=v_0^{s,0}-f_1, \quad \partial_{\zeta} f_1 = f , \quad
-\tilde{f}\triangleq\partial_t f_1(t,\zeta) +\frac{\varphi'(t)}{\varphi(t)} f_1(t,\zeta)$$.

From \eqref{vs1}
we know that $w(t,\zeta)$ satisfies the following problem in $\{ t>0,\zeta\in R\}:$
\begin{equation}\label{wss1}
\begin{cases}
\partial_t w- \mu(\varphi(t))^2 \frac{1}{\tau^{I,0}(t,h)}(\partial^2_{\zeta}w+\partial_{\zeta} f) +\frac{\varphi'(t)}{\varphi(t)}w
 =\tilde{f},\\
w(t,\pm \infty)=0,\\
w(0,\zeta) = 0,
\end{cases}
\end{equation}
using the similarly argument of $v^{B,0}$, we obtain the following  estimate for $0< T_2\leq T_1,$
$$\langle\zeta\rangle^n v_{\pm}^{s,0}(t,\zeta)\in \bigcap_{k+[\frac{l}{2}]\leq s-2} (W^{k,\infty}(0,T_2;H^l)\cap H^k(0,T_2;H^{l+1})).$$
Remark : here $f_1$ has no decay at $  \infty $ but $\partial_{\zeta} f_1 $ has.\\

Then we finally obtain that
$$\langle\zeta \rangle^n (\tau_{\pm}^{s,j},v_{\pm}^{s,j})(t,\zeta) \in \bigcap_{k+[\frac{l}{2}]\leq s+1-3j}\left(W^{k,\infty}(0,T_2;H^l)\bigcap H^k(0,T_2;H^{l+1}) \right),$$
$$\langle\zeta \rangle^n u_{\pm}^{s,j+1} \in \bigcap_{k+[\frac{l+1}{2}]\leq s+1-3j}\left(W^{k,\infty}(0,T_2;H^l)\cap H^k(0,T_2;H^{l+1}) \right),$$
for all $1\leq j\leq \frac{s+1}{3}$.

\section{Stability of approximate solutions}
From above sections , we know that  if the initial data $U^{\epsilon}_{0}(x)= (\tau_0,u_0,v^{\epsilon}_0)^{T}(x)$ satisfies Assumption, $U_{0} -(\bar{\tau},0,0)^{T} $ be piecewise $ H^{s}$ with $s> 3M-1$ for a fixed integer $ M\geq 1$, $\bar{\tau}$ given as before, and the compatibility conditions hold for problems of profiles $\{U^{I,j},U^{B,j},U^{s,j}\}_{0\leq j\leq M}$, then $ U^{a}(t,x)$ defined by
\begin{align}\label{eq128}
U^{a}(t,x)=(\rho^{a},u^{a},v^{a})^{T}(t,x)= \Sigma_{j=0}^{ M} \epsilon^{\frac{j}{2}}\left(U^{I,j}(t,x)+U^{B,j}(t,\frac{x}{\sqrt{\epsilon}})+U^{s,j}(t,\frac{x-h}{\sqrt{\epsilon}})\right)
\end{align}
is an approximate solution to the problem \eqref{eqZ2} for $0\leq t\leq T^{*}$, in the sense that
\begin{equation}\label{eq129}
\begin{cases}
\mathcal{L}(U^a)= R^{\epsilon},\\
u^a(t,0)=0,\\
U^a(0,x)= U^{\epsilon}_0(x),
\end{cases}
\end{equation}
where  in the coefficient matrix of $\mathcal{L}$,
$$ r^a(t,x)= r_0(x) + \int_0^t u^a(s,x)ds,\ \ \ with \ r_0(x)=\sqrt{a^2 +2\int_0^x \tau_0(y)dy}$$
and the source term $R^{\epsilon}$ satisfies
\begin{align}\label{eq130}
\|R^{\epsilon}\|^2_{L^2(0,T^*;L^2)} +\epsilon^2 \|\partial_xR^{\epsilon}\|^2_{L^2(0,T^*;L^2)}\leq C\epsilon^M.
\end{align}
Moreover, recall Assumption, there exists a positive constant, and we still denoted by $\tilde{\tau}$, such that
 $$ \tilde{\tau}\geq\tau^a (t,x)  > 0,\forall t\in[0,T^*], \forall x\geq 0,$$
and  we have
$$r^a(t,x)\geq \underline{r}>0 \ \ \ \forall t\in[0,T^*], \forall x\geq 0$$
for a positive constant $\underline{r}$ .
\begin{remark}
{\color {red} we can prove that $U^a\in H^2$, only need to check that $U^a$ and $\partial_x U^a$ has no jump  at $x=h$.

From the equations of profile $\{U^{I,j},U^{s,j}\}$ , and boundary condition \eqref{B1},
we know that, at $x=h$:
\begin{align*}
&[U^a]= \sum_{0\leq j\leq M} \epsilon^{\frac{j}{2}}\left\{[U^{I,j}+U^{s,j}]\right\} =0,\\
&[\partial_x U^a]=\frac{1}{\sqrt{\epsilon}}[\partial_\zeta U^{s,0}]+ \sum_{0\leq j\leq M-1} \epsilon^{\frac{j}{2}}\left\{[\partial_x U^{I,j} + \partial_\zeta U^{s,j+1}]\right\}=0.
\end{align*}
%
%
}
\end{remark}
The main result of the paper is the following one.

\begin{theorem}\label{T1}
Let the initial data  $U^{\epsilon}_{0}(x)= (\tau_0,u_0,v^{\epsilon}_0)^{T}(x)$ satisfies Assumption, $U_{0} -(\bar{\tau},0,0)^{T} $ be piecewise $H^{s+2}$ with $s> 8$. Let $U^a(t,x)$ be an approximate solution to the problem \eqref{eqZ2} given in \eqref{eq128} and satisfy \eqref{eq129} with \eqref{eq130} holding for $M\geq3$. Then there exists $C> 0$ independent of $\epsilon$ such that there exists a unique solution $U^{\epsilon}(t,x)$ to \eqref{eqZ2} such that $U^{\epsilon}-U^a \in C([0,T^*];H^1)$ and
\begin{equation}\label{th1}
\|U^{\epsilon}-U^a\|^2_{L^{\infty}(0,T^*;L^2)} +\epsilon^2\|\partial_x U^{\epsilon}-\partial_x U^a\|^2_{L^{\infty}(0,T^*;L^2)}\leq C\epsilon^M.
\end{equation}
Moreover, we have
\begin{equation}\label{th2}
U^{\epsilon}-\left(U^{I,0}(t,x)+U^{B,0}(t,\frac{x}{\sqrt{\epsilon}})+U^{s,0}(t,\frac{x-h}{\sqrt{\epsilon}})\right) = O(\sqrt{\epsilon}) \ \  in\  \  L^{\infty}([0,T^*]\times \mathbf{R}^+).
\end{equation}
\end{theorem}

{\color{red}
\begin{remark}
Since  $ \tau^{B,0}=u^{B,0}= \tau^{s,0}= u^{s,0}=0 $ , so we neither assume that the strength of layers suitable small, nor introduce the weighted norm as in \cite{LW}.
\end{remark}
}

We shall follow the authors' idea \cite{LW} to prove this theorem.

\subsection{Estimate of errors}
For the approximate solution $U^a(t,x)$ given in \eqref{eq128}, denote by $U(t,x)= U^{\epsilon}(t,x)-U^a(t,x) = (\tau,u,v)^{T}(t,x)$, and let
$$r(t,x)= \int_0^t u(s,x)ds.$$
we  will derive the problem of the error $U(t,x) $ and then study estimates of $U(t,x)$ by the energy method.
From \eqref{eqZ2} and the problem \eqref{eq129}, we know that $U(t,x)$ satisfies the following problem:
\begin{equation}\label{eq131}
\begin{cases}
\partial_t U + A(U^a)\partial_x U +dA(U^a)\cdot U\partial_x U^a -\epsilon B(U^a)\partial^2_x U + R^{l} = N,\\
U(0,x)=0,\\
U_{II}(t,0)= 0
\end{cases}
\end{equation}
where
$$R^{l}= \epsilon(Q_2(U^a +U)(\partial_x U^a, \partial_x U)+ Q_2(U^a +U)(\partial_x U, \partial_x U^a)),$$
and $ N=R^q - R^{\epsilon}$, with
\begin{equation*}
\begin{split}
R^q =& -\left( A(U^a +U)\partial_x (U^a + U) - A(U^a)\partial_x U^a - A(U^a)\partial_x U -dA(U^a)\cdot U \partial_x U^a\right)\\
     &+\epsilon \left( B(U^a +U)\partial^2_{x}(U^a +U) - B(U^a)\partial^2_{x}U^a -B(U^a)\partial^2_{x}U\right)\\
     &-\left(Q_1(U^a +U)(U^a + U, U^a +U) - Q_1(U^a)(U^a, U^a)\right)\\
     &-\epsilon \left((Q_2(U^a + U)-Q_2(U^a))(\partial_x U^a, \partial_x U^a) + Q_2(U^a +U)(\partial_x U, \partial_x U)\right)\\
     &-\epsilon\left(Q_3(U^a +U, U^a +U)(U^a +U,U^a +U)- Q_3(U^a)(U^a, U^a)\right)\\
     &+\epsilon\left( V(U^a +U,\partial_x(U^a +U))- V(U^a,\partial_x U^a)\right).
\end{split}
\end{equation*}

The local existence and uniqueness  of a smooth solution to \eqref{eq131} is followed by the classical theory; see \cite{FR2,SZ}, for instance. So the main task to prove the Theorem is  to show the estimate \eqref{th1}, Once \eqref{th1} is proved, the estimate \eqref{th2} follows immediately by using Sobolev embedding theorem.

Define
\begin{equation}\label{eq132}
T^{\epsilon}= sup\{ T\in [0,T^*];\ \ \text{ such that}\ \   E(t)\leq \epsilon^p \ \ \ \forall t\in [0,T]\},
\end{equation}
where
$$E(t)= \|U(t)\|^2 + \epsilon^2\|\partial_xU(t)\|^2,$$
and $p\leq M$, will be chosen later; the notation $\|\cdot\|$ denotes the standard  $L^2-$ norm in the x-variable.
\\
To prove that $T^{\epsilon}= T^*$ , it suffices to check by the an energy estimate that we cannot have $E(T^{\epsilon})= \epsilon^p$. 
{\color {red} So the main idea of the proof is to deduce that the following  energy}
estimate for the solution of \eqref{eq131}:
\begin{align}\label{eq133}
E(t)\leq C \left(\epsilon^M + \int_0^t E(s)ds \right) \ \ \ \ \forall t\in [0,T^{\epsilon}],
\end{align}
where $C> 0 $ is a constant independent of $\epsilon$ and $T^\epsilon$ . Once \eqref{eq133} is proved , Theorem follows by a classical argument.
\\
As $sup_{0\leq t\leq T^\epsilon} E(t)\leq \epsilon^p$, by using the Sobolev embedding, we have for $p\geq 1$ and $t\in [0,T^{\epsilon}]$,
$$\|U(t)\|^2_{L^\infty}\leq C\epsilon^{p-1} \leq C.$$
Using this a priori bound, let us estimate the term $R^q$ given in \eqref{eq131},
using the same process as in \cite{LW}.
First we rewrite $R^q$ as
$$R^q = \sum_{i=1}^5 R_{i}^{q},$$
where
\begin{align*}
\begin{split}
R_1^q =& -\left( A(U^a +U)-A(U^a)\right)\partial_x U,\\
R_2^q =& -(A(U^a +U)-A(U^a)- dA(U^a)\cdot U)\partial_x U^a,\\
R_3^q =& \epsilon ( B(U^a +U)-B(U^a))\partial^2_{x}U,\\
\begin{split}
R_4^q =&\epsilon ( B(U^a +U)-B(U^a))\partial^2_{x}U^a\\
&-(Q_1(U^a +U)(U^a + U, U^a +U) - Q_1(U^a)(U^a, U^a))
\end{split}\\
\begin{split}
R_5^q =&-\epsilon( (Q_2(U^a +U)-Q_2(U^a))(\partial_x U^a, \partial_x U^a) +Q_2(U^a +U)(\partial_x U,\partial_x U))\\
&-\epsilon\left(Q_3(U^a +U, U^a +U)(U^a +U,U^a +U)- Q_3(U^a)(U^a, U^a)\right)\\
&+\epsilon\left( V(U^a +U,\partial_x(U^a +U))- V(U^a,\partial_x U^a)\right)
\end{split}
\end{split}
\end{align*}
Thus, for $p\geq 1$ and for all $s\in[0,t]$ with fixed $t\leq T^\epsilon$,
 denote by $U_{II}=(u,v)^T$ we have
\begin{align*}
\begin{split}
\|R_1^q (s)\|^2\lesssim & (\|r(s)\|^2_{L^{\infty}}+\|U(s)\|^2_{L^\infty})\cdot \|\partial_x U(s)\|^2\lesssim \epsilon^{p-1}\|\partial_x U(s)\|^2, \\
\|R_2^q (s)\|^2\lesssim & \frac{1}{\epsilon}(\|r(s)\|^2_{L^{\infty}}+\|U(s)\|^2_{L^\infty})\cdot\|U(s)\|^2 \lesssim \epsilon ^{p-2}\|U(s)\|^2,\\
\|R_3^q (s)\|^2\lesssim & \epsilon^2(\|r(s)\|^2_{L^{\infty}}+\|U(s)\|^2_{L^\infty})\cdot \|\partial^2_x U_{II}(s)\|^2\lesssim \epsilon^{p+1}\|\partial^2_x U_{II}(s)\|^2,\\
\|R_4^q (s)\|^2\lesssim & \epsilon^2 \{\frac{1}{\epsilon^2} (\|U(s)\|^2+\|r(s)\|2)\} + (\|r(s)\|^2_{L^{\infty}}+\|U(s)\|^2_{L^\infty})\cdot \|U(s)\|^2 + \|U(s)\|^2 \\
&\lesssim (\|U(s)\|^2+\|r(s)\|2) +\epsilon^{p-1}\|U(s)\|^2, \\
\begin{split}
\|R_5^q (s)\|^2\lesssim & \epsilon^2 \{\frac{1}{\epsilon^2} (\|U(s)\|^2+\|r(s)\|2)+ \|\partial_x U(s)\|^2 +\|U(s)\|^2 \\
&+(\|r(s)\|^2_{L^{\infty}}+\|U(s)\|^2_{L^\infty})\cdot \|U(s)\|^2\}\\
&\lesssim \|U(s)\|^2+\|r(s)\|2 +\epsilon^2 \|\partial^2_xU(s)\|^2 + \epsilon^{p+1}\|U(s)\|^2 +\epsilon^2\|U(s)\|^2
\end{split}
\end{split}
\end{align*}
As we know that$u^{B,0}=u^{s,0} \equiv 0$, it follows that
$$\|\partial_x u^a(s,\cdot)\|_{L^{\infty}}, \|\partial_x r^a(s,\cdot)\|_{L^{\infty}}\leq C ,\quad \forall t \in [0,T^*]$$
for a constant $ C> 0 $ independent of $\epsilon$. Thus,using the spacial structure of $A$ and the above estimate, we have
\begin{align*}
\begin{split}
\epsilon^2\|\partial_x(R^q)_1(s)\|^2\lesssim &\epsilon^2 \{ \|\partial^2_x U_{II}\|^2\cdot(\|r(s)\|^2_{L^{\infty}}+\|U(s)\|^2_{L^\infty} )\\
&+ \|\partial^2_x U_{II}\|^2_{L^\infty}\cdot (\partial_x \|r(s)\|^2+\partial_x\|U(s)\|^2)\\
&+\frac{1}{\epsilon}(\|r(s)\|^2+\|U(s)\|^2) + (\partial_x \|r(s)\|^2+\partial_x\|U(s)\|^2)\}\\
\lesssim &\epsilon^2 (\partial_x \|r(s)\|^2+\partial_x\|U(s)\|^2) + \epsilon (\|r(s)\|^2+\|U(s)\|^2)\\
&+\epsilon^{p+1}\|\partial^2_x U_{II}(s)\|^2 + \epsilon^{p-1}(\partial_x \|r(s)\|^2+\partial_x\|U(s)\|^2).
\end{split}
\end{align*}
where $(R^q)_1$ stands by the first component of $R^q$. In the above estimates and the following calculations, we denote by $\lesssim, O(1),$ and $C$ generic numbers,possibly large,which do not depend on $\epsilon$ and $T^\epsilon$.\\
Note
$$\|\partial_x r(s)\|= \|\int_0^s \partial_x u(\tau,x)d\tau\|_{L^2_x} \lesssim\|\|\partial_xu(\tau,x)\|_{L^2_x}\|_{L^1_{\tau}(0,s)} \lesssim \|\partial_x U\|_{L^2(0,t;L^2)}.$$
Then the following estimates holds:
 \begin{lemma}\label{L1}
Under the assumption of Theorem \ref{T1}, let $U(t,x)$ be the solution to \eqref{eq131}, and $T^\epsilon$ given in \eqref{eq132}. Then, the following estimate of the source term $N$ given in \eqref{eq131} holds for all $t\in[0,T^\epsilon]$:
\begin{align}\label{E1}
\int_0^t (\|N(s)\|^2 +\epsilon^2 \|\partial_x N_1(x)\|^2) ds
\lesssim \epsilon^M + \int_0^t \left(E(s) + \epsilon^{p-3}E(s) +\epsilon^{p+1}\|\partial^2_xU_{II}(s)\|^2\right)ds
\end{align}
\\
where $C$ is a positive constant independent of $\epsilon$ and $T^\epsilon$.
\end{lemma}

We split the proof of Theorem\ref{T1} in various lemmas.
we observed the following fact first.
\begin{proposition}\label{p6}
There exists a positive definite symmetric matrix $S(U^a)$ such that
\item[\rm(i)] the matrix $SA(U^a)$ is symmetric,
\item[\rm(ii)] the matrix $SB(U^a)$ is symmetric and there exists a positive constant $c_0$ such that
$$SB(U^a)X\cdot X \geq c_0|X_{II}|^2,$$
where $X_{II} $ stands for $(x_2,x_3)^T$ for any vector $X=(x_1,x_2,x_3)^T\in \mathbf{R}^3$.
\end{proposition}
It is easy to check that the matrix
$$S(U^a)= (r^a)^{-1} \cdot diag\{(\frac{1}{\tau^a})^{\gamma+1},1,1\}$$
is the desired one which satisfies the above proposition.

\begin{lemma}
The matrix $S$ is positive definite, and $SA(u_0)$ is symmetric for some state $u_0$, if and only if there exist a matric $L$ composed of the left eigenvectors of $A(u_0)$ with $LA= diag\{\lambda_1,\cdots,\lambda_m\}L$, such that $S=L^{T} L,$ where $\lambda_i ,\quad i=1,\cdots, m$, are eigenvalues of $A(u_0)$.
\end{lemma}

\begin{lemma}\label{L2}
Under the assumption of Theorem \ref{T1}, there exist positive constant $C_1$  and $M_1$ independent of $\epsilon$ and $T^\epsilon$, such that the solution $U(t,x)$ to \eqref{eq131} satisfies the estimate for $t \in [0,T^\epsilon]$,
\begin{align}\label{E2}
\|U(t)\|^2 + \int_0^t \epsilon\|\partial_x U_{II}(s)\|^2ds \leq C_1 \int_0^t (\|N(s)\|^2 + \|U(s)\|^2 + \epsilon M_1 \|\partial_x\tau(s)\|^2)ds .
\end{align}
\end{lemma}
\begin{proof}
Notice the special structure of the matrix $A$, the third component $v(t,x) $ can be estimated first. The remained components of $U(t,x)$ denoted by $w(t,x)=(\tau,u)^T(t,x)$. Correspondingly, the associated components of $U^a$ denote by $w^a$.

\item[\rm(1)] From the problem \eqref{eq131}, we know that $v(t,x)$ satisfies the following initial-boundary value problem:
\begin{align}\label{eq134}
\begin{cases}
\partial_t v - \epsilon\mu\frac{ (r^a)^2}{\tau^a} \partial^2_x v +\epsilon\mu (\frac{r^a}{\tau^a})^2(\partial_x \tau^a \partial_x v + \partial_x v^a\partial_x \tau)
 = N_3,\\
v(t,0)= v(0,x)=0,\\
\end{cases}
\end{align}
where $N_3$ is the third component of $N$ given in\eqref{eq131}.
Multiply \eqref{eq134} by $v$ and integrate the resulting equation with respect to $x$ variable over $ [0,\infty)$ to obtain that
\begin{align}
\frac{1}{2}\frac{d}{dt}\|v(t)\|^2 -\epsilon\mu \left(\frac{(r^a)^2}{\tau^a} \partial^2_x v, v \right)
+\epsilon\mu \left((\frac{r^a}{\tau^a})^2[\partial_x \tau^a \partial_x v + \partial_x v^a\partial_x \tau],v\right)= (N_3,v).
\end{align}
It is easy to have
\begin{align}
\begin{split}
-\epsilon\mu \left(\frac{(r^a)^2}{\tau^a} \partial^2_x v, v \right)
&= \epsilon\mu \left(\frac{(r^a)^2}{\tau^a} \partial_x v,\partial_x v \right) + \epsilon \mu \left(\partial_x\left(\frac{(r^a)^2}{\tau^a} \right)\partial_x v,v \right)\\
&\geq \epsilon \mu c_1 \|\partial_x v\|^2 - C_2\sqrt{\epsilon}\|\partial_x v\|\cdot \|v\|\\
&\geq\frac{1}{2}\epsilon\mu c_1 \|\partial_x v(t)\|^2 - C_3\|v(t)\|^2,
\end{split}
\end{align}
where $c_1= (\underline{r})^2/\tilde{\tau} $  and $C_i> 0 \ \ (i=2,3)$ is independent of $\epsilon$. And
\begin{align}
\begin{split}
-&\epsilon\mu \left((\frac{r^a}{\tau^a})^2[\partial_x \tau^a \partial_x v + \partial_x v^a\partial_x \tau],v\right)\\
&\lesssim \sqrt{\epsilon}\|\partial_x v\|\cdot\|v\| + \sqrt{\epsilon}\|\partial_x \tau \|\cdot\|v\|
\leq \frac{1}{4}\epsilon\mu c_1 \|\partial_x v\|^2 + C_4\|v\|^2 + \epsilon M_1 \|\partial_x \tau\|^2,
\end{split}
\end{align}
for a constant $C_4>0$ independent of $\epsilon$.
Then
$$\frac{1}{2}\frac{d}{dt}\|v(t)\|^2 + \frac{1}{4}\mu c_1 \epsilon\|\partial_x v(t)\|^2 \leq C \|v(t)\|^2 + \epsilon M_1\|\partial_x \tau(t)\|^2 + C \|N_3\|^2, $$
by using Gronwall inequality that
\begin{align}\label{E2-1}
\|v(t)\|^2 +\epsilon\int_0^t \|\partial_x v(s)\|^2 ds \leq C \int_0^t \|N_3\|^2 + \epsilon M_1\|\partial_x \tau(s)\|^2 ds.
\end{align}
\item[\rm(2)] To estimate $w(t,x)$, from the problem \eqref{eq131} we have the following initial-boundary value problem :
\begin{equation}\label{e2}
\begin{cases}
\partial_t w + A_1(U^a)\partial_x w +dA_1(U^a)\cdot w\partial_x w^a -\epsilon B_1(U^a)\partial^2_x w +R^s =N^s,\\
w(0,x)= 0,\\
w(t,0)= 0,
\end{cases}
\end{equation}
where $R^s, N^s $ are the terms $R^l, N$ given in \eqref{eq131} but without the  third component, and
\[
A_1(U^a)=
\left(
\begin{array}{cc}
0, & -r^a \\
-r^a(\tau^a)^{-1-\gamma}, & 0
\end{array}
\right),\ \quad
\\
B_1(U^a)= \frac{(r^a)^2}{\tau^a} \cdot diag\{0,\lambda +2 \mu\}
\]
In the lemma \ref{L6} below we will show that there exists a positive constant $C_2$ independent of $\epsilon$ and $T^\epsilon$, such that the solution $w(t,x)$ to \eqref{e2} satisfies the following estimate for $ t\in [0,T^\epsilon]$:
\begin{align}\label{E2-2}
\|w(t)\|^2 + \epsilon \int_0^t \|\partial_x u(s)\|^2ds \leq C_2 \int_0^t \|N(s)\|^2 + \|U(s)\|^2 + \epsilon M_1 \|\partial_x\tau\|^2 ds
\end{align}
Combining \eqref{E2-1} with \eqref{E2-2}, one can obtain \eqref{E2} when $ M_1$ is properly small.
\end{proof}
Next we try to verify the estimate\eqref{E2-2}. We observe that Proposition \ref{p6} still holds for the matrixes $A_1(U^a),B_1(U^a)$ given in problem \eqref{e2} with respect to the symmetrizer
$$S_1(U^a)= (r^a)^{-1} \cdot diag\{(\frac{1}{\tau^a})^{\gamma+1},1\}$$
Denote every function $f(U^a)$ by $f^a$, and the left,right eigenvectors of $ A_1(U^a) $ by $L^a_i, R^a_i,$ respectively, with the normalization $L^a_iR^a_j= \delta_{ij}$.\\
Set $L^a=((L^a_1)^T,(L^a_2)^T)^T$ and $R^a=(R^a_1,R^a_2)$. then we have
$$L^aA^a_1R^a = D^a \triangleq  diag\{-c^a,c^a\},\quad S^a_1= (L^a)^T L^a \quad \text{and} \ \ L^aR^a= I$$
\\
To estimate $w(t,x)$  diagonalize the equation of \eqref{e2} by setting $w=R^aV $ in \eqref{e2} and get
\begin{align}\label{eq135}
\begin{split}
\partial_t V +D^a \partial_x V + L^a (\partial_t R^a\cdot V +& A^a_1\partial_x R^a \cdot V + dA^a_1\cdot R^a V\cdot \partial_x w^a)\\
&- \epsilon L^sB^a_1 \partial^2_x(R^aV)+L^a R^s =L^aN^s.
\end{split}
\end{align}
By using \eqref{eq135} we have the following result, which can be used to obtain the estimate \eqref{E2-2}.
\begin{lemma}\label{L6}
Under the assumption of the Theorem \ref{T1}, there exist positive constants $C, \ \ M_1$, independent of  $\epsilon$ and $T^\epsilon$, such that the solution$w(t,x)$ to problem \eqref{e2} satisfies the estimate for all $t\in[0,T^\epsilon]$,
\begin{align}\label{e3}
\|w(t)\|^2 + \epsilon \int_0^t \|\partial_x u(s)\|^2ds \leq C_2 \int_0^t \|N(s)\|^2 + \|U(s)\|^2 + \epsilon M_1 \|\partial_x\tau\|^2 ds
\end{align}
\end{lemma}
\begin{proof}
Multiplying \eqref{eq135} by $V$ and integrating the resulting equation with respect to $x$ in $[0,\infty)$, obtain that
\begin{align}
\begin{split}
\frac{1}{2}\frac{d}{dt}\|V\|^2 +(D^a\partial_x V,V) + (L^a(A^a_1\partial_xR^a\cdot V + dA^a_1\cdot R^aV \partial_x w^a),V)\\
- \epsilon(L^aB^a_1\partial^2_x(R^aV),V)+ (L^a(\partial_tR^a\cdot V- R^s),V) = (L^aN^s,V).
\end{split}
\end{align}
Each term in above can be treat as follows.
\begin{align*}
(D^a\partial_x V,V)&= -\frac{1}{2}(\partial_x D^a\cdot V,V)+ \frac{1}{2}(D^aV,V)|_{x=0}
=-\frac{1}{2}(\partial_x D^a\cdot V,V)+ \frac{1}{2}(L^aA^a_1R^aV,V)|_{x=0}\\
&=-\frac{1}{2}(\partial_x D^a\cdot V,V)+ \frac{1}{2}(L^aA^a_1R^a V,L^a W)|_{x=0}\\
&=-\frac{1}{2}(\partial_x D^a\cdot V,V)+ \frac{1}{2}(S^a_1A^a_1 w,w)|_{x=0}\\
&=-\frac{1}{2}(\partial_x D^a\cdot V,V)
\end{align*}
then
$$(D^a\partial_x V,V)= -\frac{1}{2}(\partial_x D^a\cdot V,V)\geq -C_7\|w\|^2$$
Next we consider the term
$$(L^a(A^a_1\partial_xR^a\cdot V + dA^a_1\cdot R^aV \partial_x w^a),V)= (L^aA^a_1{\color {blue}\partial_xR^a}\cdot V,V)+ (L^a dA^a_1\cdot R^aV {\color {blue}\partial_x w^a}),V)\leq C_8\|w\|^2$$
It remains to estimate the term $\epsilon(L^aB^a_1\partial^2_x(R^aV),V)$.
\begin{align*}
\begin{split}
-& \epsilon(L^aB^a_1\partial^2_x(R^aV),V)= -\epsilon(S^a_1 B^a_1\partial^2_x w,w)\\
&= -\epsilon(\partial_x(S^a_1 B^a_1\partial_x w),w)+ \epsilon(\partial_x(S^a_1 B^a_1)\partial_x w,w)\\
&=\epsilon(S^a_1 B^a_1\partial_x w,\partial_x w)+ \epsilon(\partial_x(S^a_1 B^a_1)\partial_x w,w)\\
&\geq c_1\epsilon \|\partial_x u\|^2- C\epsilon \|\partial_x u\|\cdot\|u\|\\
&\geq c_1 \epsilon\|\partial_x u\|^2 - C\|u\|^2,
\end{split}
\end{align*}
It is easy to have
\begin{align*}
\begin{split}
&(L^a\partial_tR^a\cdot V,V)\leq C\|w\|^2,\\
&- (L^aR^s,V)= -(L^aR^s,L^a w)\leq C\epsilon \|u\|\cdot\|\partial_x w\|\leq C\sqrt{\epsilon} \|u\|\cdot\|\partial_x w\|\leq C\|u\|^2 + \epsilon M_1\|\partial_x w\|^2,
\end{split}
\end{align*}
then we get \eqref{e3} by choosing $M_1$ small.
\end{proof}
To close an energy estimate from \eqref{E2}, the term $\epsilon \int_0^t \|\partial_x \tau(s)\|^2 ds$ need to be control. This is the aim of following lemma.
\begin{lemma}\label{L3}
Under the same assumption as  in Theorem \ref{T1}, for any $\alpha>0$, there exists a constant $C(\alpha)>0 $, independent of $\epsilon$ and $T^\epsilon$, such that for all $t\in [0,T^\epsilon]$,
\begin{align}\label{E3}
\begin{split}
&\epsilon^2 \|\partial_x \tau (t)\|^2
-\epsilon \left( \frac{1}{(\lambda +2\mu)}\frac{\tau^a}{r^a} u(t), \partial_x \tau(t) \right)
+ \epsilon \int_0^t \|\partial_x \tau (s)\|^2 ds \\
&\leq C(\alpha)\int_0^t \epsilon \|\partial_x u(s)\|^2 + \epsilon^2\|\partial_x N_1(s)\|^2 + \epsilon \|N(s)\|^2 + \|U(s)\|^2 ds + \alpha\epsilon \int_0^t \|\partial_t \tau(s)\|^2 ds.
\end{split}
\end{align}
\end{lemma}
\begin{proof}
At first, we take the derivative of the first equation of \eqref{eq131} with respect to the spacial variable $x$, this yields
\begin{align}\label{E3-1}
\partial_t(\partial_x \tau) + r^a\partial^2_x u + \partial_x r^a\partial_x u= \partial_x N_1,
\end{align}
and then, by using the second equation of the system\eqref{eq131} to express $\partial^2_x u$, obtain that
\begin{align}\label{E3-2}
\begin{split}
\partial^2_x u =& \frac{1}{(\lambda +2\mu)\epsilon}\frac{\tau^a}{(r^a)^2}(\partial_t u - r^a(\frac{1}{\tau^a})^{\gamma+1}\partial_x \tau - N_2)\\
&+ O(1)(\partial_x u + \partial_x \tau)+ \frac{O(1)}{\epsilon}(|r| +|\tau|),
\end{split}
\end{align}
and hence we get
\begin{align}\label{E3-3}
\begin{split}
&\partial_t(\partial_x \tau) - \frac{1}{(\lambda +2\mu)\epsilon}\frac{\tau^a}{ r^a}(\partial_t u - r^a(\frac{1}{\tau^a})^{\gamma+1}\partial_x \tau)\\
&= \partial_x N_1 + \partial_x(r^a)\partial_x u - O(1)\frac{N_2}{\epsilon}
- O(1)(\partial_x u + \partial_x \tau)- \frac{O(1)}{\epsilon}(|r| +|\tau|).
\end{split}
\end{align}

We take the scalar product of \eqref{E3-3} with $\epsilon^2 \partial_x \tau $
get that
\begin{align}\label{E3-4}
\begin{split}
&\epsilon^2 \|\partial_x \tau(t)\|^2 + \frac{c_1\epsilon}{2} \int_0^t \|\partial_x \tau(s)\|^2 ds
- \int_0^t \epsilon \left( \frac{1}{(\lambda +2\mu)}\frac{\tau^a}{r^a}\partial_t u(s), \partial_x \tau(s)\right) ds\\
&\leq C\int_0^t \epsilon^2 \|\partial_x N_1(s)\|^2 + \epsilon \|N_2(s)\|^2+  \epsilon\|U\|^2 + \epsilon \|\partial_x u(s)\|^2 ds.
\end{split}
\end{align}

It remains to estimate the term
$\int_0^t \epsilon \left( \frac{1}{(\lambda +2\mu)}\frac{\tau^a}{r^a}\partial_t u(s), \partial_x \tau(s)\right) ds.$
Since
\begin{align*}
\begin{split}
&\epsilon\left( \frac{1}{(\lambda +2\mu)}\frac{\tau^a}{r^a}\partial_t u, \partial_x \tau \right)\\
=&
\epsilon\partial_t \left( \frac{1}{(\lambda +2\mu)}\frac{\rho^a}{r^a} u, \partial_x \rho \right)-
\epsilon \left(\frac{1}{(\lambda +2\mu)}\partial_t(\frac{\tau^a}{r^a})\cdot u, \partial_x \tau \right)
-\epsilon \left(\frac{1}{(\lambda +2\mu)}\frac{\phi}{h}\frac{\tau^a}{r^a} u, \partial_t \partial_x \tau \right),\\
\end{split}\\
and\\
\begin{split}
-\epsilon \left(\frac{1}{(\lambda +2\mu)}\frac{\tau^a}{r^a} u, \partial_t \partial_x \tau \right)=
\epsilon \left(\frac{1}{(\lambda +2\mu)}\partial_x(\frac{\tau^a}{r^a})\cdot u, \partial_t \tau \right)
+\epsilon\left( \frac{1}{(\lambda +2\mu)}\frac{\tau^a}{r^a}\partial_x u, \partial_t \tau \right),
\end{split}
\end{align*}
where the last equation holds by $U_{II}(t,0)=0$.
Thus, using $U(0,x) =0 $ obtain that ,
\begin{align*}
\begin{split}
\epsilon & \int_0^t \left( \frac{1}{(\lambda +2\mu)}\frac{\tau^a}{r^a}\partial_t u(s), \partial_x \tau(s)\right)
-\left( \frac{1}{(\lambda +2\mu)}\frac{\tau^a}{r^a}\partial_x u(s), \partial_t \tau(s)\right) ds\\
&=\epsilon \left( \frac{1}{(\lambda +2\mu)}\frac{\tau^a}{r^a} u(t), \partial_x \tau(t) \right)
-\epsilon \int_0^t \left( \frac{1}{(\lambda +2\mu)}\partial_t(\frac{\tau^a}{r^a})\cdot u(s), \partial_x \tau(s)\right) ds\\
&+\epsilon \int_0^t \left( \frac{1}{(\lambda +2\mu)}\partial_t(\frac{\tau^a}{r^a})\cdot u(s), \partial_t \tau(s)\right) ds
\end{split}
\end{align*}
Then
\begin{align}\label{E3-5}
\begin{split}
\epsilon & \int_0^t \left( \frac{1}{(\lambda +2\mu)}\frac{\tau^a}{r^a}\partial_t u(s), \partial_x \tau(s)\right) ds
-\epsilon \left( \frac{1}{(\lambda +2\mu)}\frac{\tau^a}{r^a} u(t), \partial_x \tau(t) \right)\\
&\leq \alpha\epsilon \int_0^t \|\partial_t \tau(s)\|^2 ds + C(\alpha)\int_0^t \epsilon \|\partial_x u(s)\|^2 ds + C\int_0^t \|u(s)\|^2 +\epsilon^2 \|\partial_x \tau(s)\|^2 ds
\end{split}
\end{align}
Finally, we get \eqref{E3} by combining \eqref{E3-4} with \eqref{E3-5}
\end{proof}
Because of \eqref{E3}, we need to estimate  $\epsilon \int_0^t \|\partial_t U(s)\|^2 ds $ .
\begin{lemma}\label{L4} Under the same assumption as given in Theorem \ref{T1}, there exists a positive constant $C_3$, independent of $\epsilon$ and $T^\epsilon$, such that for all $ t\in [0,T^\epsilon]$,
\begin{align}\label{E4}
\epsilon^2 \|\partial_x U_{II}(t)\|^2 + \epsilon \int_0^t \|\partial_t U(s)\|^2ds \leq C_3 \int_0^t \epsilon \|N(s)\|^2 + \epsilon\|\partial_x U(s)\|^2 + \|U(s)\|^2 ds.
\end{align}
\end{lemma}
\begin{proof}
Multiplying $ S^a\partial_t U$ on both side of the equations \eqref{eq131},then
\begin{align}\label{E4-1}
\begin{split}
\int_0^t \|\partial_t U(s)\|^2 ds - \epsilon \int_0^t \left(S^aB^a\partial^2_x U(s),\partial_t U(s) \right) ds\\
\leq C\int_0^t \|N(s)\|^2 + \frac{1}{\epsilon}\|U(s)\|^2 + \|\partial_x U(s)\|^2  ds
\end{split}
\end{align}
According to the special structure of $S^aB^a$ and $\partial_t U_{II}(t,0)=0$, using an integration by parts that
\begin{align}\label{E4-2}
\begin{split}
\epsilon(S^aB^a\partial^2_x U(s),\partial_t U(s)) &= -\epsilon (S^aB^a\partial_x U, \partial_t\partial_x U)
 - \epsilon (\partial_x(S^aB^a)\partial_x U, \partial_t U)\\
&=-\frac{\epsilon}{2}\partial_t (S^aB^a\partial_x U,\partial_x U) +O(1)\epsilon (\|\partial_xU\|\|\partial_t U\|+ \|\partial_x U_{II}\|^2)
\end{split}
\end{align}
Therefore, \eqref{E4} obtained by plugging \eqref{E4-1} into \eqref{E4-2} and using Proposition \ref{p6}
\end{proof}

\subsection{End of the proof of Theorem \ref{T1}}

At this stage, we have all the elements need to  prove Theorem \ref{T1}.\\
First, by adding \eqref{E3} to $2\alpha\cdot \eqref{E4}$ and choosing $\alpha$ small,
\begin{align}\label{M1}
\begin{split}
\epsilon^2\|\partial_x\tau(t)\|^2& + 2\alpha \epsilon^2\|\partial_x U_{II}(t)\|^2
-\epsilon \left( \frac{1}{(\lambda +2\mu)}\frac{\tau^a}{r^a} u(t), \partial_x \tau(t) \right)\\
&+ \epsilon \int_0^t \|\partial_x\tau(s)\|^2 +\alpha \epsilon \int_0^t \|\partial_t U(s)\|^2 ds\\
&\leq C_2 \int_0^t \epsilon \|N(s)\|^2 + \|U(s)\|^2 + \epsilon^2 \|\partial_x N_1(s)\|^2 + \epsilon \|\partial_x U_{II}(s)\|^2 ds,
\end{split}
\end{align}
for a constant $C_2>0$.
Second, by calculating  $ \eqref{E2} +4CM_1C_2\cdot \eqref{M1}$
\begin{align}\label{M2}
\begin{split}
\|U(t)\|^2& + 4CM_1C_2\left\{ \epsilon^2 (\|\partial_x\tau(t)\|^2 + 2\alpha\|\partial_x U_{II}(t)\|^2)
-\epsilon \left(\frac{1}{(\lambda +2\mu)}\frac{\tau^a}{r^a} u(t), \partial_x \tau(t) \right)\right\}\\
&+\epsilon \int_0^t (1-4M_1CC_2)\|\partial_x U_{II}(s)\|^2 + 2CM_1\|\partial_x\tau(s)\|^2 + 4CM_1\alpha \|\partial_t U(s)\|^2 ds\\
&\leq C_5 \int_0^t \|U(s)\|^2 +\|N(s)\|^2 +\epsilon^2 \|\partial_x N_1(s)\|^2 ds.
\end{split}
\end{align}
Note that
$$4CM_1C_2\epsilon \left( \frac{1}{(\lambda +2\mu)}\frac{\tau^a}{r^a} u(t), \partial_x \tau(t) \right)\\
\leq 2C M_1C_2\|U(t)\|^2 + 2CM_1C_2 \epsilon^2\|\partial_x\tau(t)\|^2,$$
So, we have
\begin{equation}
\begin{split}
\|U(t)\|^2 + 4CM_1C_2 \epsilon^2 \|\partial_x\tau(t)\|^2
&-4CM_1C_2\epsilon \left(\frac{1}{(\lambda +2\mu)}\frac{\tau^a}{r^a} u(t), \partial_x \tau(t) \right)\\
&\geq (1-2C M_1C_2)\|U(t)\|^2 + 2CM_1C_2 \epsilon^2\|\partial_x\tau(t)\|^2.
\end{split}
\end{equation}
Thus choose $ M_1$ small satisfying $1-4M_1CC_2 >0$ and $1-2C M_1C_2 >0$,
combining\eqref{M1} and \eqref{M2} obtain that
\begin{align}\label{M3}
E(t) + \epsilon \int_0^t \|\partial_x U(s)\|^2 + \|\partial_t U(s)\|^2 ds
\leq C \int_0^t \|U(s)\|^2 + \|N(s)\|^2+ \epsilon^2\|\partial_x N_1\|^2 ds,
\end{align}
with $E(t) = \|U(t)\|^2 + \epsilon^2\|\partial_xU(t)\|^2.$

To conclude Theorem \ref{T1}, from \eqref{E1} and \eqref{M3}, it remains to estimate $\int_0^t \|\partial^2_x U_{II}(s)\|^2 ds$, we just use the equation \eqref{eq131}, we have
\begin{align}\label{M4}
\begin{split}
\epsilon^3 \int_0^t \|\partial^2_x U_{II}(s)\|^2 ds
&\leq C\int_0^t \epsilon (\|\partial_t U(s)\|^2 + \|\partial_x U(s)\|^2 + \|N(s)\|^2) + \|U(s)\|^2 ds
\end{split}
\end{align}
Thus, by substituting \eqref{E1} and \eqref{M3} into \eqref{M4} deduced that
\begin{align}
\begin{split}
E(t) + &\epsilon\int_0^t \left(\|\partial_x U(s)\|^2 +\|\partial_t U(s)\|^2\right) ds\\
&\leq C\epsilon^M + C_3 \int_0^t E(s) + \epsilon^{p-3}E(s) + \epsilon^{p-1} \left(\|\partial_t U(s)\|^2 + \|\partial_x U(s)\|^2 \right )ds
\end{split}
\end{align}
Hence, for $p\geq3$ and $\epsilon$ small , we have
$$ E(t) + \epsilon\int_0^t \|\partial_x U(s)\|^2 +\|\partial_t U(s)\|^2 ds \leq C\epsilon^M + C \int_0^t E(s) ds ,$$
where $C>0$ is independent of $\epsilon$ and $T^\epsilon$. Last, by using Gronwall inequality, we obtain that
$$E(t)\leq C\epsilon^M \quad \forall t\in [0,T^\epsilon].$$
Noting that
$$ U^a(t,x) = U^{I,0}(t,x) + U^{B,0}(t,\frac{x}{\sqrt{\epsilon}})+U^{s,0}(t,\frac{x-h}{\sqrt{\epsilon}}) +O(\sqrt{\epsilon}),$$
and then, from the estimate \eqref{th1},  using Sobolev embedding theorem, we obtain \eqref{th2} and complete the proof of Theorem \ref{T1}.
\\

\section*{Acknowledgments} The authors would like to thank Q. Zhao and C.J. Liu for stimulating discussions.










\begin{thebibliography}{99}
\pagestyle{plain} \linespread{0.3}

\bibitem{Alex}
\newblock R.Alexander,Y.G. Wang, C.J. Xu, T,Yang ,
\newblock Well-posedness of the Prandtl equation in sobolev spaces,
\newblock J. Amer. Math. Soc., \textbf{28} (2015), pp. 745--784.

\bibitem{C}
\newblock A. Corli ,
\newblock Asymptotic analysis of contact discontinuities,
\newblock  Ann. Mat. Pura Appl.  \textbf{173} (4) (1997), pp. 163--202.

\bibitem{WE2}
\newblock W.E, B.Engquist ,
\newblock Blow up of solutions of the unsteady Prandtl's equation,
\newblock Comm. Pure Appl. Math., \textbf{50} (1997), pp. 1287-- 1293.

\bibitem{G3}
\newblock D.G\'{e}rard-Varet, E. Dormy ,
\newblock On the ill-posedness of the Prandtl equation,
\newblock J. Amer. Math. Soc., \textbf{23} (2010), pp. 591-- 609.

\bibitem{GX}
\newblock J. Goodman, Z.P. Xin ,
\newblock Viscosity limits for piecewise smooth solutions to systems of conservation laws,
\newblock Arch. Ration. Mech. Anal., \textbf{121} (3) (1992),  pp. 235 -- 265.

\bibitem{G5}
\newblock E. Grenier ,
\newblock On the nonlinear instability of Euler and Prandtl equations,
\newblock Comm. Pure Appl. Math., \textbf{53} (2000), pp. 1067-- 1091.

\bibitem{Guo}
\newblock Y. Guo, T. Nguyen ,
\newblock A note on the Prandtl boundary layers
\newblock Comm. Pure Appl. Math., \textbf{64} (2011), pp. 1416 -- 1438.

\bibitem{Hong}
\newblock L. Hong, J.K. Hunter ,
\newblock Singularity formation and instability in the unsteady  inviscid and viscous Prandtl equations,
\newblock Commun. Matn. Sci., \textbf{1} (2003), pp. 293 -- 316.

\bibitem{HLM}
\newblock F.M. Huang, J. Li, A. Matsumura ,
\newblock Asymptotic stability of combination of viscous contact wave with rarefaction waves for one-dimensional compressible Navier-Stokes system ,
\newblock Arch. Ration. Mech. Anal., \textbf{197} (1) (2010), pp. 89 -- 116.

\bibitem{HMX}
\newblock F.M. Huang, A. Matsumura, Z.P. Xin ,
\newblock Stability of contact discontinuities for the 1-D compressible Navier-Stokes equations,
\newblock Arch. Ration. Mech. Anal., \textbf{179} (1) (2006), pp. 55 -- 77.

\bibitem{HW}
\newblock P. Hartman, A. Wintner
\newblock On hyperbolic partial differential equations,
\newblock Amer. J. Math., \textbf{74} (1952), pp. 834--864.

\bibitem{HWY}
\newblock F.M. Huang, Y. Wang, T. Yang ,
\newblock Vanishing viscosity limit of the compressible Navier-Stokes equations for solutions to a Riemann problem,
\newblock Arch. Ration. Mech. Anal., \textbf{203} (2) (2012), pp. 379 -- 413

\bibitem{LW}
\newblock C.J. Liu, Y.G. Wang ,
\newblock stability of boundary layers for the nonisentropic compressible circularly symmetric 2D flow,
\newblock SIAM. J. Math. Anal., \textbf{46} (1) (2014),  pp.256 -- 309.

\bibitem{M}
\newblock S.X. Ma,
\newblock Zero dissipation limit to strong contact discontinuity for the 1-D compressible Navier¨CStokes equations,
\newblock J. Dfferential Equations ,\textbf{248} (2010), pp. 95 --110.


\bibitem{AM}
\newblock A.Majda ,
\newblock Compressible fluid flow and systems of conservation laws in several space variables,
\newblock Springer-Verlag, New York, 1984.

\bibitem{Mas1}
\newblock N.Masmoudi ,
\newblock The Euler limit of the Navier-Stokes equations,and rotating fluids with boundarty,
\newblock Arch. Ration. Maech. Anal., \textbf{142} (4) (1998), pp. 375 -- 394.

\bibitem{Mas}
\newblock N. Masmoudi, T.K. Wong ,
\newblock Local-in-time existence and uniqueness of solutions to the Prandtl equations by energy methods ,
\newblock  Comm. Pure Appl. Math. \textbf{68} (2015), pp. 1683--1741.

\bibitem{Ole}
\newblock O.A. Oleinik, V.N. Samokhin ,
\newblock Mathematical Models  in Boundary Layers Theory,
\newblock Chapman \& Hall/ CRC, Boca Raton, FL, 1999 .

\bibitem{Prandtl}
\newblock L.Prandtl ,
\newblock \"{U}ber Fl\"{u}ssigkeitsbewegungen bei sehr kleiner Reibung,
\newblock in Verh. Int. Math. Konger.,Heideberg, Germany, (1904), pp. 484--494.

\bibitem{FR}
\newblock F.Rousset ,
\newblock Characteristic boundary layers in real vanishing viscosity limits,
\newblock J. Differential Equations, \textbf{210} (2005) , pp. 25 -- 64.

\bibitem{FR2}
\newblock F.Rousset ,
\newblock Stability of small amplitude boundary layers for mixed hyperbolic-parabolic systems,
\newblock Trans. Amer. Math. Soc., \textbf{355} (2003),  pp. 2991 -- 3008.

\bibitem{Sam1}
\newblock M. Sammartino, R.E. Caflish ,
\newblock Zero viscosity limit for analytic solutions of the Navier-Stokes equation on a half-space, I. existence for Euler and Prandtl equations,
\newblock Comm. Math. Phys., \textbf{192} (1998), pp. 433 -- 461.

\bibitem{Sam2}
\newblock M. Sammartino , R.E. Caflish ,
\newblock Zero viscosity limit for analytic solutions of the Navier-Stokes equation on a half-space, II. construction of the Navier-Stokes solutions,
\newblock Comm. Math. Phys., \textbf{192} (1998), pp. 463 -- 491.

\bibitem{SZ}
\newblock D. Serre, K. Zumbrun ,
\newblock Boundary layers stability in real vanishing viscosity limit,
\newblock Comm. Math. Phys., \textbf{221} (2001), pp. 267 -- 292.

\bibitem{MET}
\newblock M.E.Taylor ,
\newblock Partial Differential Equations III, Nonlinear Equations,
\newblock Springer-Verlag, New York, 1996.

\bibitem{TW}
\newblock M. Temam, X.M. Wang ,
\newblock boundary layers associated with incompressible Navier-Stokes equations: The noncharacteristic boundary case,
\newblock J. Diffential Equations, \textbf{179} (2002), pp. 647 -- 686.

\bibitem{WW}
\newblock Y.G. Wang,  M. Williams,
\newblock The inviscid limit and stability of characteristic boundary layer for the compressible Navier-Stokes equations with Navier-friction boundary condtions,
\newblock Ann.Inst. Fourier (Grenble), \textbf{62} (2012), pp. 2257 -- 2314.

\bibitem{WX}
\newblock Y.G. Wang, Z.P. Xin ,
\newblock Zero-viscosity limit of the linearized compressible Navier-Stokes equations with highly oscillatory forces in the half-plane,
\newblock SIAM. J. Math. Anal., \textbf{37} (2006), pp. 1256 -- 1298.

\bibitem{X1}
\newblock Z.P. Xin ,
\newblock Zero dissipation limit to rarefaction waves for the one-dimensional Navier¨CStokes equations of compressible isentropic gases,
\newblock Comm. Pure Appl. Math., \textbf{46} (5) (1993), pp. 621 -- 665.

\bibitem{XY}
\newblock Z.P. Xin, T. Yanagisawa ,
\newblock Zero-viscosity limit of the linearized Navier-Stokes equations for a comprssible viscous fluid in the half-plane,
\newblock Comm. Pure Appl. Math., \textbf{52} (1999), pp. 479 -- 541

\bibitem{zp}
\newblock Z.P. Xin, L. Zhang ,
\newblock On the global existence of solutions to the Prandtl's system,
\newblock Adv. Math., \textbf{181} (2004), pp 88 -- 133.

\bibitem{Y}
\newblock S.H.Yu,
\newblock Zero-dissipation limit of solution with shocks for systems of hyperbolic conservation laws,
\newblock Arch. Ration. Mech. Anal., \textbf{146} (4) (1999), pp. 275 -- 370.

\bibitem{ZHH}
\newblock H.H. Zeng ,
\newblock Stability of a superposition of shock waves with contact discontinuities for systems of viscous conservation laws ,
\newblock J. Differential Equations \textbf{246} (2009), pp. 2081 -- 2102.

\bibitem{ZWT}
\newblock Y.H. Zhang, R.H. Pan, Y. Wang, Z. Tan ,
\newblock Zero dissipation limit with two interacting shocks of the 1D non-isentropic Navier-Stokes equation,
\newblock Indiana Univ. Math. J., \textbf{62} (1) (2013), pp. 249 --309.

\end{thebibliography}
\end{document}